\documentclass[reqno]{amsart}
\usepackage[english]{babel}
\usepackage{amscd,amssymb,amsmath,amsfonts,latexsym,amsthm}
\usepackage{ulem}
\usepackage{inputenc}
\usepackage{graphicx,color}
\newtheorem{thm}{Theorem}[section]
\newtheorem{cor}[thm]{Corollary}
\newtheorem{lem}[thm]{Lemma}
\newtheorem{prop}[thm]{Proposition}
\newtheorem{defn}[thm]{Definition}
\newtheorem{exam}[thm]{Example}

\numberwithin{equation}{section}

\DeclareMathOperator{\Hom}{Hom}

\begin{document}

\title[On maximal solvable extensions of a pure non-characteristically nilpotent Lie algebra]{On maximal solvable extensions of a pure non-characteristically nilpotent Lie algebra}

\author{K.K. Abdurasulov, B.A. Omirov}

\address{Kobiljon K.Abdurasulov \newline \indent
Institute of Mathematics of Uzbekistan Academy of Sciences, Tashkent, Uzbekistan} \email{{\tt abdurasulov0505@mail.ru}}

\address{Bakhrom A.Omirov \newline \indent
National University of Uzbekistan $\&$ Institute of Mathematics of Uzbekistan Academy of Sciences, \newline \indent
Tashkent, Uzbekistan} \email{{\tt omirovb@mail.ru}}

\begin{abstract} In this paper we introduce the notion of pure non-characteristically nilpotent Lie algebra and under a condition we prove that a complex maximal extension of a finite-dimensional pure non-characteristically nilpotent Lie algebra is isomorphic to a semidirect sum of the nilradical and its maximal torus. We also prove that such solvable Lie algebras are complete and we specify a subclass of the maximal solvable extensions of pure non-characteristically nilpotent Lie algebras that have trivial cohomology group. Some comparisons with the results obtained earlier are given.
\end{abstract}



\maketitle


\section{Introduction}

Lie algebras play an essential role in mathematics and physics, especially in quantum field theory, quantum mechanics and string theory. An extensive study of Lie algebras gave many beautiful results and generalizations. From the classical theory of finite-dimensional Lie algebras, it is known that an arbitrary Lie algebra is decomposed into a semidirect sum of the solvable radical and semisimple subalgebra (Levi's theorem). According to the Cartan-Killing theory, a semisimple Lie algebra can be represented as a direct sum of simple ideals, which are completely classified \cite{Jac}. Thanks to Mal’cev's and Mubarakzjanov's  results the study of non-nilpotent solvable Lie algebras is reduced to the study of nilpotent ones \cite{Mal}, \cite{mubor}. Therefore, the study of finite-dimensional Lie algebras is focused on nilpotent algebras. In fact, the main idea of the fundamental Mal’cev's work consists of introducing the so-called splitable type of solvable Lie algebras. He proved that these algebras can be written as a semidirect sum of the nilradical and an abelian subalgebra, whose elements in the regular representation are diagonalizable.

The study of solvable Lie algebras with some special types of nilradical was the subject of various papers (see \cite{AnCaGa1,AnCaGa2,BoPaPo,Cam,NdWi,SnKa,SnWi,TrWi,WaLiDe} and references therein). For example,  the cases where the nilradical is filiform, quasi-filiform and Abelian  were considered in \cite{AnCaGa2,Cam,NdWi,WaLiDe}.

Since the description of finite-dimensional nilpotent Lie algebras is an immense problem, It is customary to study them under some additional restrictions. In the case of non-nilpotent solvable Lie algebras, the picture is more or less clear. The point here is that such algebras can be reconstructed by using their nilradical and special types of derivations. Mubarakzjanov's method (see  \cite{mubor}) on construction of solvable Lie algebras in terms of their nilradicals seemed to be effective.

It seems that if one writes a non-nilpotent solvable Lie algebra as a direct sum of a nilradical and the subspace complementary to the nilradical then the estimate of the dimension of complementary subspace is the key points of the method. In \cite{snoble} it was proved that this dimension is bounded above by the number of the generators of the nilradical.

The latest result (see \cite{Qobil}) in the use of this relation asserts that if the dimension equals to the number of generators of the nilradical such a solvable Lie algebra is unique up to isomorphism. Moreover, this solvable Lie algebra is complete (centerless and the first cohomology group is trivial).

It is well known that Lie algebras form an algebraic variety defined by the skew-symmetry and the Jacobi identity, where the class of solvable Lie algebras is a subvariety. The key role in the variety of Lie algebras play the so-called rigid algebras, they generate the irreducible components of the variety. It is known that the vanishing of the second cohomology group of Lie algebras implies their rigidity \cite{rigid}. The natural problem arises to calculate the second cohomology groups of algebras with coefficients in the adjoint representation, which is very important due to its applications in the deformation theory. The results obtained so far on descriptions of some solvable Lie algebras with the maximal possible complementary subspace show that they have trivial second cohomology groups. So, it might seems that these groups for such algebras are always trivial. However, counterexamples constructed recently show that it is not the case. Therefore, the structure of such solvable Lie algebras is complicated, and their study must be carried out more delicately.

The description of the cohomology groups of solvable Lie algebras with the diagonal action of the complementary subspace to the nilradical can be realized by using the Hochschild-Serre factorization theorem (see \cite{cohomology}). This theorem reduces the  calculation of the higher cohomology groups to that of lower cohomologies of certain subspaces of the solvable Lie algebra \cite{Ancochea1}, \cite{Leger2}.

Recall that in terms of the cohomology groups the first Whitehead's lemma is stated as follows: For a finite-dimensional semisimple Lie algebra over a field of characteristic zero and a finite-dimensional $\mathcal{L}$-module $\mathcal{M}$ one has $H^1(\mathcal{L},\mathcal{M})=0.$ The validity of this lemma for non-nilpotent solvable Lie algebras is unknown, while for an arbitrary nilpotent Lie algebra $\mathcal{N}$ it is known that $H^1(\mathcal{N},\mathcal{N})\neq0.$

This paper is devoted to study of maximal solvable extensions of some nilpotent Lie algebras and their cohomology groups. In fact, we obtain the description of some maximal solvable extensions of nilpotent Lie algebras that gives a partial positive solution to \v{S}nobl's Conjecture  \cite{Snobl}, which states the following:

Let $\mathfrak{n}$ be a complex nilpotent (not characteristically nilpotent) Lie algebra and $\mathfrak{s}, \tilde{\mathfrak{s}}$ be solvable Lie algebras with the nilradical $\mathfrak{n}$ of maximal dimension in the sense that no such solvable algebra of larger dimension exists. Then, $\mathfrak{s}$ and $\tilde{\mathfrak{s}}$ are isomorphic.

In fact, as was shown by V.Gorbatsevich in \cite{Gorbatsevich} the conjecture is not true in that general form as asserted by \v{S}nobl. Namely, V.Gorbatsevich constructed two
non-isomorphic solvable extensions of the nilradical, which is a direct sum of a seven-dimensional characteristically nilpotent and an one-dimensional nilpotent Lie algebras. Based on these constructions we conclude that \v{S}nobl's conjecture is not true when the nilradical is not pure non-characteristically nilpotent algebra (see Definition \ref{pure}).

The organization of the paper is as follows. Section 2 contains some preliminary results.  In Section 3 we present known examples which show the existence of non-isomorphic maximal extensions of some nilpotent Lie algebras and give the notion of pure non-characteristic nilpotent Lie algebras. Under a condition the explicit description of complex maximal solvable Lie algebras with a given pure non-characteristically nilpotent nilradical is obtained (see Theorem \ref{mainthm} and Proposition \ref{lie}). As a consequence of this description we derive that such maximal solvable extensions  are isomorphic to a semidirect sum of the nilradical and its maximal torus (Theorem \ref{torus}).

In Section 4 it is established that a complex maximal solvable extension of pure non-characteristically nilpotent Lie algebra which satisfies the condition {\bf A)}, admits only inner derivations (Theorem \ref{thmH1eq0}) whereas non-maximal solvable extension admits an outer derivation (Proposition \ref{thmH1neq0}).

In section 5 we give several applications of the description of maximal solvable extensions of nilpotent Lie algebras (Theorem \ref{mainthm}, Proposition \ref{lie}) to some classification results obtained earlier. Section 6 is devoted to the proof of vanishing of the cohomology groups for some classes of maximal solvable extensions of nilpotent Lie algebras (Theorem 6.2, Theorem 6.3).
Finally, in Section 7 we exhibit some examples showing that application of Theorem 6.2 much simplifies the calculations of the cohomology groups of some maximal solvable extensions of nilpotent Lie algebras.

Throughout the paper all vector spaces and algebras considered to be finite-dimensional over the field $\mathbb{C}$ unless stated otherwise.

\section{Preliminaries}

Let $N$ be a nilpotent Lie algebra and $\mathcal{T}$ be its torus. Consider a solvable Lie algebra $R_{\mathcal{T}}=\mathcal{N}\oplus \mathcal{T}$ with the products
$$[\mathcal{N},\mathcal{N}], \quad [\mathcal{N}, \mathcal{T}]=\mathcal{T}(N), \quad [\mathcal{T}, \mathcal{T}]=0.$$
In the case of maximal torus $\mathcal{T}_{max}$ of $N$ we denote the solvable Lie algebra $\mathcal{N}\oplus\mathcal{T}_{max}$ by $R_{\mathcal{T}_{max}}.$

Let $\{e_1, \dots, e_n\}$ be a basis of the nilpotent Lie algebra $\mathcal{N}$ then $[e_i,e_j]=\sum_{t=1}^{n}c_{i,j}^te_t$ with $c_{i,j}^t$ in the ground field. For $i, j, t$ such that $c_{i,j}^t\neq 0$ we consider the system of linear equations
$$S_e: \quad \left\{\alpha_{i}+\alpha_{j}=\alpha_{t}\right.$$
in the variables $\alpha_1, \dots, \alpha_n$ as $i, j, t$ run from $1$ to $n$. Following the paper \cite{Leger1} we denote by $r\{e_1, \dots, e_n\}$ the rank of the system $S_e.$

Denote $r\{\mathcal{L}\}=min \  r\{e_1, \dots, e_n\}$ as $\{e_1, \dots, e_n\}$ runs over all bases of $\mathcal{N}$.

Note that for a nilpotent Lie algebra $\mathcal{N}$ over an algebraically closed field the equality $dim \mathcal{T}_{max}=dim \mathcal{N} -r\{\mathcal{L}\}$ holds true \cite{Leger1}.

\begin{defn}
Derivations $d_1, d_2, \dots, d_t$ of a Lie algebra $\mathcal{L}$ are said to be nil-independent, if  the nilpotency of the derivation $\alpha_1 d_1 + \alpha_2 d_2 + \dots + \alpha_n d_n$ implies $\alpha_i=0, 1 \leq i \leq t$.
\end{defn}

There is an estimate of the codimension of the nilradical in a solvable Lie algebra as follows (see \cite{mubor}).
\begin{thm} \label{thm22} Let $\mathcal{R}$ be a solvable Lie algebra and $\mathcal{N}$ be its nilradical. Then the dimension of the subspase complementary to $\mathcal{N}$ is not greater than the  maximal number of nil-independent derivations of $\mathcal{N}$.
\end{thm}

Recall that for a Lie algebra $\mathcal{L}$ and a $\mathcal{L}$-module $\mathcal{M}$ the notion of the group of cohomologies of $\mathcal{L}$ is given as follows. Consider the spaces
$$CL^0(\mathcal{L},\mathcal{M}):= \mathcal{M}, \quad CL^n(\mathcal{L},\mathcal{M})=\Hom(\wedge^{n} \mathcal{L}, \mathcal{M}), \ n > 0.$$
and coboundary operators $d^n : C^n(\mathcal{L},\mathcal{M}) \rightarrow C^{n+1}(\mathcal{L},\mathcal{M})$  ($\mathbb{F}$-homomorphisms) defined by
 $$(d^n\varphi)(x_1, \dots , x_{n+1}): = \sum\limits_{i=1}^{n+1}(-1)^{i+1}[x_i, \varphi(x_1,\dots, \widehat{x}_i, \dots , x_{n+1})]$$
$$+\sum\limits_{1\leq i<j\leq {n+1}}(-1)^{i+j}\varphi([x_i,x_j], x_1, \dots,\widehat{x}_i, \dots, \dots , \widehat{x}_j, \dots ,x_{n+1}),$$
where $\varphi\in C^n(\mathcal{L},\mathcal{M})$ and $x_i\in \mathcal{L}$. The property $d^{n+1}\circ d^n=0$ leads that the derivative operator $d:=\sum\limits_{i \geq 0}d^i$ satisfies the property
$d\circ d = 0$. Therefore, the $n$-th cohomology group is well defined by
$$H^n(\mathcal{L},\mathcal{M}): = Z^n(\mathcal{L},\mathcal{M})/ B^n(\mathcal{L},\mathcal{M}),$$
where the elements $Z^n(\mathcal{L},\mathcal{M}):=Ker d^{n+1}$ and $B^n(\mathcal{L},\mathcal{M}):=Imd^n$ are called {\it $n$-cocycles} and {\it $n$-coboundaries}, respectively.

The space $H(\mathcal{L},\mathcal{M})=\bigoplus\limits_{n>0}H^n(\mathcal{L},\mathcal{M})$ is called {\it the cohomology group of the Lie algebra $\mathcal{L}$ relative to the module $\mathcal{M}$}.

The computation of the cohomology groups is laborious, the Hochschild - Serre factorization theorem below much simplifies their computations for semidirect sums of algebras \cite{Leger2}.

\begin{thm} \label{thmSerre} Let $\mathcal{R} =\mathcal{N}\oplus \mathcal{Q}$ be a solvable Lie algebra such that $\mathcal{Q}$ is abelian and the operators $ad_x$ ($x\in \mathcal{Q}$) are diagonal. Then adjoint cohomology $H^p (\mathcal{R},\mathcal{R})$ satisfies the following isomorphism
\begin{equation}\label{eq222}
H^p(\mathcal{R},\mathcal{R})\simeq\sum\limits_{a+b=p}^{}H^a(\mathcal{Q},\mathbb{F})\otimes H^b(\mathcal{N},\mathcal{R})^{\mathcal{Q}},
\end{equation}
where
\begin{equation}\label{eq14}
H^b(\mathcal{N},\mathcal{R})^{\mathcal{Q}}=\{ \varphi\in H^b(\mathcal{N},\mathcal{R}) \ | \ (x.\varphi)=0,\ \ x\in \mathcal{Q}\}
\end{equation}
is the space of $\mathcal{Q}$--invariant cocycles of $\mathcal{N}$ with values in $\mathcal{R}$,
the invariance being defined by:
\begin{equation}\label{eq333}
(x.\varphi)(e_1,\ldots,e_b)=[x,\varphi(e_1,\ldots,e_b)]-
\sum\limits_{s=1}^{b}\varphi(e_1,\ldots,[x,e_s],\ldots,e_b).
\end{equation}
\end{thm}
Observe that $H^a(\mathcal{Q},\mathbb{F})=\bigwedge ^a \mathcal{Q},$ hence $H^p(\mathcal{R},\mathcal{R})$ vanishes if and only if $H^b(\mathcal{N},\mathcal{R})^{\mathcal{Q}}=0$ for $0\leq b \leq p$.

\section{Description of some maximal solvable extensions of nilpotent Lie algebras}\label{sec3}

Let us present first an example which shows that in \v{S}nobl's conjecture the condition of the ground field to be $\mathbb{C}$ is essential.
P\begin{exam} Consider three-dimensional Heisenberg Lie algebra
\begin{equation}
\label{exm1} H_1: \quad [e_2,e_3]=-[e_3, e_2]=e_1.
\end{equation}
Due to Theorem \ref{mainthm} up to isomorphism there exists a unique complex solvable five-dimensional Lie algebra with the nilradical $H_1$ and the table of multiplications
\begin{equation}\label{exm2}
[e_2,e_3]=e_1,\ [e_2,x_1]=e_2,\ [e_1,x_1]=e_1,\ [e_3,x_2]=e_3,\ [e_1,x_2]=e_1.
\end{equation}

However, in the paper \cite{mubor} it was showed the existence of the following two non-isomorphic to each other real solvable Lie algebras:
$$\emph{g}_{5,36}: \ [e_2,e_3]=e_1,\ [e_1,x_1]=e_1,\ [e_2,x_1]=e_2,\ [e_2,x_2]=-e_2,\ [e_3,x_2]=e_3,$$
$$\emph{g}_{5,37}: \ [e_2,e_3]=e_1,\ [e_1,x_1]=2e_1,\ [e_2,x_1]=e_2,\ [e_3,x_1]=e_3,\ [e_2,x_2]=-e_3,\ [e_3,x_2]=e_2.$$
\end{exam}

Here are two non-isomorphic maximal solvable extensions of a given niradical from \cite{Gorbatsevich}.
\begin{exam} Let $\mathcal{R}_1$ and $\mathcal{R}_2$ be solvable Lie algebras with the following tables of multiplications:
$$\mathcal{R}_1: \ \left\{\begin{array}{lll}
[e_1,e_2]=e_3, & [e_2,e_3]=e_6, &\\[1mm]
[e_1,e_3]=e_4, & [e_2,e_4]=e_7, & \\[1mm]
[e_1,e_4]=e_5, & [e_2,e_5]=e_7, \\[1mm]
[e_1,e_5]=e_6, & [e_3,e_4]= -e_7, \\[1mm]
[e_1,e_6]=e_7, & [e_8,x]=e_8,\\[1mm]
\end{array}\right. \quad \mathcal{R}_2: \
\left\{\begin{array}{lll}
[e_1,e_2]=e_3, & [e_2,e_3]=e_6, & \\[1mm]
[e_1,e_3]=e_4, & [e_2,e_4]=e_7, \\[1mm]
[e_1,e_4]=e_5, & [e_2,e_5]=e_7, & \\[1mm]
[e_1,e_5]=e_6, & [e_3,e_4]= -e_7, \\[1mm]
[e_1,e_6]=e_7, & [e_2,x]=e_7,\\[1mm]
               & [e_8,x]=e_8.\\[1mm]     \end{array}\right.$$

It is easy to check that the solvable Lie algebras $\mathcal{R}_1$ and $\mathcal{R}_2$ have split nilradical, which is a direct sum of characteristically nilpotent ideal with the basis $\{e_1, \dots, e_7\}$ and one-dimensional ideal with the basis $\{e_8\}$. Clearly, the algebra $\mathcal{R}_1$ is the direct sum of $2$-dimensional solvable ideal $<x, e_8>$ and $7$-dimensional characteristically nilpotent ideal $\mathcal{N}_7=<e_1, \dots, e_7>$ (it is obtained by deriving $ad(x)$), while $\mathcal{R}_2$ is the extension corresponding to deriving $ad(x) + d$, where $d(e_2)=e_7$ is the outer derivation of $\mathcal{N}_7$. Due to $13=dimDer(\mathcal{R}_1) > dimDer(\mathcal{R}_2)=12$ these solvable Lie extensions are not isomorphic. Thus, we obtain that \v{S}nobl's conjecture is not true in general.
\end{exam}

In order to exclude the nilradicals for which there exist non-isomorphic maximal solvable extensions, we introduce the notion of pure non-characteristically nilpotent algebra.

\begin{defn} \label{pure} An nilpotent Lie algebra is called pure non-characteristically nilpotent algebra if either it is non-split and non-characteristic nilpotent or in  the case of split it does not contain characteristically nilpotent direct factors.
\end{defn}
For characteristically nilpotent Lie algebras we refer to the survey paper \cite{Ancochea111}.

So, starting from now all nilpotent Lie algebras are assumed pure non-characteristically nilpotent Lie algebras.

Since the square of a solvable Lie algebra is nilpotent, any solvable Lie algebra has non-trivial nilradical. Therefore, for a solvable Lie algebra $\mathcal{R}$ we have a decomposition $\mathcal{R}=\mathcal{N}\oplus \mathcal{Q}$, where $\mathcal{N}$ is the nilradical of $\mathcal{R}$ and $\mathcal{Q}$ is the subspace complementary to $\mathcal{N}$.

Due $[ad_x,ad_y]=ad_{[x,y]}$ for any $x,y\in \mathcal{R}$ the solvability of $\mathcal{R}$ is equivalent to the solvability of the subalgebra $ad(\mathcal{R})$ of $Der(\mathcal{N})$. Applying Lie's Theorem \cite{Jac} to the algebra $ad(\mathcal{R})$, we conclude that there exists a basis of $\mathcal{R}$ such that all operators ${ad_{z}}, \ z\in \mathcal{R}$ have upper-triangular matrix form. Particular, all operators $ad_{x|\mathcal{N}}$ $(x\in \mathcal{Q})$ in the basis of $\mathcal{N}$ have upper-triangular matrix form too. Note that the basis $\{e_1,\dots,e_k,e_{k+1},\dots,e_n\}$ of $\mathcal{N}$ (here $e_i, \ 1\leq i \leq k$ are generators of $\mathcal{N}$) can be chosen as follows: any non generator basic element is the right-normed word of the alphabet $\{e_1, \dots, e_k\}$. Such a basis of $\mathcal{N}$ is called {\it the natural basis}.

Further the basis $\{e_1, \dots, e_n$\} means the natural basis of $\mathcal{N}$.

In fact, according to \cite{Gant} for the operator $ad_{{x}|\mathcal{N}}, \ x\in \mathcal{Q}$ one has a decomposition $ad_{{x}|\mathcal{N}}=d_{0}+d_{1},$
where $d_{0}$ is diagonalizable and $d_{1}$ is nilpotent derivations of $\mathcal{N}$ such that $[d_{0}, d_{1}]=0$.

Let set
$$ad_{{x}|\mathcal{N}}=\left(\begin{array}{cccccccc}
\alpha_{1}&*&*&\dots &*&*\\[1mm]
0&\alpha_{2}&*&\dots &*&*\\[1mm]
0&0&\alpha_{3}&\dots &*&*\\[1mm]
\vdots&\vdots&\vdots&\ddots &\vdots&\vdots\\[1mm]
0&0&0&\dots &\alpha_{n-1}&*\\[1mm]
0&0&0&\dots &0&\alpha_{n}\\[1mm]
\end{array}\right).$$

Due to the upper-triangularity of $ad_{{x}|\mathcal{N}}$ the matrix of $d_{0}$ has the diagonal form
$d_{0}=diag(\alpha_{1},  \dots, \alpha_{n}).$

The Leibniz's rule for derivation $d_0$ implies the restrictions on diagonal elements $\alpha_{i}$ with $1\leq i\leq n:$
$$\sum c_{i,j}^t\alpha_{t}e_t=\sum c_{i,j}^td_{0}(e_t)=d_{0}(\sum c_{i,j}^te_t)=d_{0}([e_i,e_j])=[d_{0}(e_i),e_j]$$
$$+[e_i,d_{0}(e_j)]=\alpha_{i}[e_i,e_j]+\alpha_{j}[e_i,e_j]=(\alpha_{i}+\alpha_{j})[e_i,e_j]=
(\alpha_{i}+\alpha_{j})\sum c_{i,j}^te_t.$$

Thus, for $i, j, t$ such that $c_{i,j}^t\neq 0$ we obtain the system of equations:
$$S_e: \quad \left\{\alpha_{i}+\alpha_{j}=\alpha_{t}.\right.$$

Now and onward we regard a solvable Lie algebra $\mathcal{R}$ as $\mathcal{R}=\mathcal{N}\oplus \mathcal{Q}$, the vector space decomposition such that the codimesion of $\mathcal{N}$ is maximal, i.e., $\mathcal{R}$ is the maximal solvable extension of the nilpotent Lie algebra $\mathcal{N}$.

\begin{lem} \label{lem111} Let $\mathcal{R}$ be the maximal solvable extension of the nilpotent Lie algebra $\mathcal{N}$. Then
$$r\{\mathcal{N}\}=r\{e_1, \dots, e_n\}=n-dim \mathcal{Q}.$$
\end{lem}
\begin{proof} Let $dim \mathcal{Q}=s$. Assume that $r\{\mathcal{N}\}\neq r\{e_1, \dots, e_n\}$. Then there exists a basis $\{f_1, \dots, f_n\}$ of $\mathcal{N}$ such that $r\{f_1, \dots, f_n\}< r\{e_1, \dots, e_n\}$ and the solutions to the system $S_f: \left\{\beta_{i}+\beta_{j}=\beta_{t}\right.$, depend on $p\ (p>s)$ free variables, say $\beta_1, \dots, \beta_p$. We get the fundamental solutions of $S_f$:
$(\beta_{1,1}, \beta_{2,1}, \dots, \beta_{n,1}), (\beta_{1,2}, \beta_{2,2}, \dots, \beta_{n,2}), \ldots, (\beta_{1,p}, \beta_{2,p}, \dots, \beta_{n,p}),$
with
$$det\left(\begin{array}{cccccc}
\beta_{1,1}&\beta_{2,1}&\dots&\beta_{p,1}\\[1mm]
\beta_{1,2}&\beta_{2,2}&\dots&\beta_{p,2}\\[1mm]
\vdots&\vdots&\ddots&\vdots\\[1mm]
\beta_{1,p}&\beta_{2,p}&\dots&\beta_{p,p}\\[1mm]
\end{array}\right)\neq 0.$$

Now we consider the solvable Lie algebra $R_{\mathcal{T}}=\mathcal{N}\oplus \mathcal{T}$ with
$$\mathcal{T}=Span\{diag(\beta_{1,1}, \beta_{2,1}, \dots, \beta_{n,1}), diag(\beta_{1,2}, \beta_{2,2}, \dots, \beta_{n,2}), \dots, diag(\beta_{1,p}, \beta_{2,p}, \dots, \beta_{n,p})\}.$$
Since $dim \mathcal{R}_{\mathcal{T}} > dim \mathcal{R}$ we get a contradiction with the assumption that $r\{\mathcal{N}\}\neq r\{e_1, \dots, e_n\}$.

Let us now assume that $r\{e_1, \dots, e_n\}\neq n-s$. This implies that the solutions to the system $S_e$ depend on $q \ (q\neq s)$ free parameters, say $\alpha_1, \alpha_2, \dots, \alpha_q$.

If $q>s$, then applying the same arguments as used above we get a contradiction with the maximality of $\mathcal{R}$.

If $q<s$, then similarly to that above we get diagonal nil-independent derivations $D_1, \dots, D_q$ such that $ad_{{x}|\mathcal{N}}-\sum\limits_{i=1}^{q}\gamma_iD_i$ is nilpotent for some $\gamma_i\in \mathbb{C}$.
Then Theorem \ref{thm22} implies that $dim \mathcal{Q}=s \leq q$, which is a contradiction.
\end{proof}

Taking into account the result of \cite{Leger1} and Lemma \ref{lem111} we derive $dim \mathcal{Q}=dim \mathcal{T}_{max}.$

We denote the free parameters in the solutions to the system $S_e$ by $\alpha_{i_1},\dots, \alpha_{i_s}$. Making renumeration the basis elements of $\mathcal{N}$ we can assume that $\alpha_{1},\dots, \alpha_{s}$ are free parameters of $S_e$. Then we get
\begin{equation}\label{eq1}
\alpha_{i}=\sum_{j=1}^{s}\lambda_{i,j}\alpha_{j},\ s+1\leq i\leq n.
\end{equation}

Let $\mathcal{Q}=Span_\mathbb{C}\{x_1,\dots,x_s\}$. Consider
$$ad_{x_i}=\left(\begin{array}{cccccccc}
\alpha_{1,i}&*&*&\dots &*&*\\[1mm]
0&\alpha_{2,i}&*&\dots &*&*\\[1mm]
0&0&\alpha_{3,i}&\dots &*&*\\[1mm]
\vdots&\vdots&\vdots&\ddots &\vdots&\vdots\\[1mm]
0&0&0&\dots &\alpha_{n-1,i}&*\\[1mm]
0&0&0&\dots &0&\alpha_{n,i}\\[1mm]
\end{array}\right), \quad 1\leq i\leq s.$$

Similarly, we have $ad_{{x_i}|\mathcal{N}}=d_{0,i}+d_{1,i},$
where $d_{0,i}=diag(\alpha_{1,i},  \dots, \alpha_{n,i})$ and $d_{1,i}$ is nilpotent derivation of $N$ such that $[d_{0,i}, d_{1,i}]=0$.

Without loss of generality, from the proof of Lemma \ref{lem111} we conclude that the vectors
$(\alpha_{1,i},  \dots, \alpha_{n,i}), \ 1\leq i \leq s$ form a basis of the fundamental solutions to the system $S_e$.

\begin{exam}\label{counter}
Consider 11-dimensional nilpotent Lie algebra $\mathcal{N}$ with the following non-zero products:
\begin{align*}
[e_2, e_1]&=e_8,   & [e_1, e_4]&=e_{10}, &[e_5,e_7]&=e_{10}, &[e_1, e_3]&=e_{9},\\
[e_5, e_6]&=e_{9}, & [e_4, e_3]&=e_{11}, &[e_7,e_6]&=e_{11}, &[e_2, e_3]&=e_{10}.
\end{align*}
One can check that a derivation of $\mathcal{N}$ has the following matrix form:
$$\left(\begin{smallmatrix}
 a_{1,1} & a_{1,2} & a_{1,3} & a_{1,4} & 0 & 0 & 0 & a_{1,8} & a_{1,9} & a_{1,10} & a_{1,11} \\
 0 & a_{2,2} & 0 & -a_{1,3} & 0 & 0 & 0 & a_{2,8} & a_{2,9} & a_{2,10} & a_{2,11} \\
 0 & 0 & a_{3,3} & a_{3,4} & 0 & 0 & 0 & a_{3,8} & a_{3,9} & a_{3,10} & a_{3,11} \\
 0 & 0 & 0 & a_{2,2}+a_{3,3}-a_{1,1} & 0 & 0 & 0 & a_{4,8} & a_{4,9} & a_{4,10} & a_{4,11} \\
 0 & 0 & 0 & 0 & a_{1,1} & a_{1,3} & a_{1,4} & a_{5,8} & a_{5,9} & a_{5,10} & a_{5,11} \\
 0 & 0 & 0 & 0 & 0 & a_{3,3} & a_{1,2}+a_{3,4} & a_{6,8} & a_{6,9} & a_{6,10} & a_{6,11} \\
 0 & 0 & 0 & 0 & 0 & 0 & a_{2,2}+a_{3,3}-a_{1,1} & a_{7,8} & a_{7,9} & a_{7,10} & a_{7,11} \\
 0 & 0 & 0 & 0 & 0 & 0 & 0 & a_{1,1}+a_{2,2} & -2 a_{1,3} & 0 & 0 \\
 0 & 0 & 0 & 0 & 0 & 0 & 0 & 0 & a_{1,1}+a_{3,3} & 0 & a_{1,3} \\
 0 & 0 & 0 & 0 & 0 & 0 & 0 & 0 & 0 & a_{2,2}+a_{3,3} & a_{2,3} \\
 0 & 0 & 0 & 0 & 0 & 0 & 0 & 0 & 0 & 0 & a_{2,2}+2 a_{3,3}-a_{1,1} \\
\end{smallmatrix}\right).$$
Let $\mathcal{R}=\mathcal{N}\oplus \mathcal{Q}$ be a maximal solvable extension of $\mathcal{N}.$ Then, $\dim \mathcal{Q}=3$ and direct computations show that $\mathcal{Q}$ admits a basis $\{x_1,x_2,x_3\}$  such that $[\mathcal{Q}, \mathcal{Q}]=0$ and
$$\begin{array}{lll}
ad_{{x_1}|\mathcal{N}}=diag(1, \ 0, \ 0,    -1, \ 1, \ 0,    -1, \ 1, \ 1, \ 0,    -1),\\[1mm]
ad_{{x_2}|\mathcal{N}}=diag(0, \ 1, \ 0, \ \ 1, \ 0, \ 0, \ \ 1, \ 1, \ 0, \ 1, \ \ 1), \\[1mm]
ad_{{x_3}|\mathcal{N}}=diag(0, \ 0, \ 1, \ \ 1, \ 0, \ 1, \ \ 1, \ 0, \ 1, \ 1, \ \ 2). \\[1mm]
\end{array}$$
Note that $\alpha_{4, i}=\alpha_{7,i}$ for any $i=1,2,3$.
\end{exam}

Let $\mathcal{R}=\mathcal{N}\oplus \mathcal{Q}$ be a maximal solvable extension of non-split  finite-dimensional nilpotent Lie algebra $\mathcal{N}$. Let $\{e_1,\dots, e_k, \dots,e_n\}$ be a natural basis of $\mathcal{N}$ (here $k$ is number of generator elements) and $\{x_1, \dots, x_s\}$ be a basis of $\mathcal{Q}$ such that $ad_{{x_i}|\mathcal{N}}$ have upper-triangular matrix forms for any $i\in \{1, \dots, s\}$, that is, for any $i\in\{1, \dots, n\}$ and $j\in \{1, \dots, s\}$ we one can write $$ad_{{x_j}|\mathcal{N}}(e_{i})=a_{i,j}e_{i}+\sum_{t=i+1}^{n}a_{i,j}^{t}e_{t}.$$

We have the following possible two cases:
\begin{itemize}
\item[1.] for any $i, j\in\{s+1,\dots, k\}, \ i\neq j$ there exists $t\in\{1,\dots, s\}$ such that
$\alpha_{i,t} \neq \alpha_{j,t}$ (see examples in Section 5);
\item[2.] there exist $i, j\in\{s+1,\dots, k\}, \ i\neq j$ such that $\alpha_{i,t} = \alpha_{j,t}$ for any $t\in\{1,\dots, s\}$ (see Example \ref{counter}).
\end{itemize}

Further we shall focus on a maximal solvable extension $\mathcal{R}=\mathcal{N}\oplus \mathcal{Q}$ of an nilpotent Lie algebra $\mathcal{N}$ with satisfy the following condition:
\begin{itemize}
\item[{\bf A})] for any $i, j\in\{s+1,\dots, k\}, \ i\neq j$ there exists $t\in\{1,\dots, s\}$ such that
$\alpha_{i,t} \neq \alpha_{j,t}.$
\end{itemize}

In order to prove the main result of this section we need the following auxiliary lemmas.

\begin{lem}\label{lem5} Let $\{u, y_1, \dots, y_{\eta}, f_1, \dots, f_{\tau}\}$ be a set of linearly independent elements of a solvable Lie algebra with the products:
$$\left\{\begin{array}{ll}
[u,y_j]=\sum\limits_{t=1}^{\tau}\beta_{j,t}f_{t},& 1\leq j\leq \eta,\\[1mm]
[f_t,y_j]=\mu_{t,j}f_t,& 1\leq j\leq \eta,\ 1\leq t\leq \tau,\\[1mm]
[u,[y_i,y_j]]=0, & 1\leq i,j\leq \eta,\\[1mm]
\end{array}\right.$$
and the coefficients $\mu_{t,j}$ satisfy the conditions:

\begin{itemize}
  \item for any $t \ (1\leq t\leq \tau)$ there exists $i \ (1\leq i \leq \eta)$ such that $\mu_{t,i}\neq0;$

\item for any $t \ (1\leq t\leq \tau)$ there exist $i,j \ \ (1\leq i\neq j \leq \eta)$ such that $\mu_{t,j}\neq \mu_{t,i}.$
    \end{itemize}
Then $$[u,y_j]=0,\ 1\leq j\leq \eta.$$
\end{lem}
\begin{proof} Let us take the change
$$u'=u-\sum\limits_{t=1,\ \mu_{t,1}\neq 0}^{\tau}\frac{\beta_{1,t}}{\mu_{t,1}}f_t.$$
Then we derive
$$[u',y_1]=\sum\limits_{t=1}^{\tau}\beta_{1,t}f_{t}-
\sum\limits_{t=1,\ \mu_{t,1}\neq 0}^{\tau}\frac{\beta_{1,t}}{\mu_{t,1}}\mu_{t,1}f_t=\sum\limits_{t=1}^{\tau}\beta_{1,t}f_{t}-\sum\limits_{t=1,\ \mu_{t,1}\neq 0}^{\tau}\beta_{1,t}f_t=\sum\limits_{t=1}^{\tau}\delta_{\mu_{t,1},0}\beta_{1,t}f_{t},$$
where $\delta_{-,-}$ is the Kronecker delta. Thus, we can assume that
$[u,y_1]=\sum\limits_{t=1}^{\tau}\delta_{\mu_{t,1},0}\beta_{1,t}f_{t}.$

For $2\leq j\leq \eta$ the following chain of equalities
$$0=[u,[y_1,y_j]]=[[u,y_1],y_j]-[[u,y_j],y_1]$$
$$=[\sum\limits_{t=1}^{\tau}\delta_{\mu_{t,1},0}\beta_{1,t}f_{t},y_j]-[\sum\limits_{t=1}^{\tau}\beta_{j,t}f_{t},y_1]
=\sum\limits_{t=1}^{\tau}\delta_{\mu_{t,1},0}\mu_{t,j}\beta_{1,t}f_{t}-
\sum\limits_{t=1}^{\tau}\mu_{t,1}\beta_{j,t}f_{t},$$
implies
$$\delta_{\mu_{t,1},0}\mu_{t,j}\beta_{1,t}=0,\ \ \mu_{t,1}\beta_{j,t}=0,\ \ 2\leq j\leq \eta,\ 1\leq t\leq \tau.$$
Taking into account the conditions of the lemma on $\mu_{t,j}$ we obtain
$$\delta_{\mu_{t,1},0}\beta_{1,t}=0,\ \ \mu_{t,1}\beta_{j,t}=0,\ \ \ 2\leq j\leq \eta,\ 1\leq t\leq \tau. $$
Hence,
\begin{equation}\label{eq33}
[u,y_1]=0,\quad \quad [u,y_j]=\sum\limits_{t=1}^{\tau}\delta_{\mu_{t,1},0}\beta_{j,t}f_{t},\ 2\leq j\leq \eta.
\end{equation}

The validity of the following
\begin{equation}\label{eq34}
[u,y_j]=0,\ 1\leq j\leq h, \quad [u,y_j]=\sum\limits_{t=1}^{\tau}\prod_{q=1}^{h}\delta_{\mu_{t,q},0}\beta_{j,t}f_{t},\ h+1\leq j\leq \eta
\end{equation}
we prove by induction on $h.$

The base of the induction is clear due to \eqref{eq33}. Assume that \eqref{eq34} are true for $h$ we prove it for $h+1.$

Consider the change
$$u'=u-\sum\limits_{t=1,\ \mu_{t,h+1}\neq 0}^{\tau}\frac{\prod_{q=1}^{h}\delta_{\mu_{t,q},0}\beta_{h+1,t}}{\mu_{t,h+1}}f_t.$$
Then due to
$$[u',y_j]=
0,\ 1\le j\leq h,$$
$$[u',y_{h+1}]=
\sum\limits_{t=1}^{\tau}\prod_{q=1}^{h}\delta_{\mu_{t,q},0}\beta_{h+1,t}f_{t}-
\sum\limits_{t=1,\ \mu_{t,h+1}\neq 0}^{\tau}\frac{\prod_{q=1}^{h}\delta_{\mu_{t,q},0}\beta_{h+1,t}}{\mu_{t,h+1}}\mu_{t,h+1}f_t$$
$$=\sum\limits_{t=1}^{\tau}\prod_{q=1}^{h}\delta_{\mu_{t,q},0}\beta_{h+1,t}f_{t}-
\sum\limits_{t=1,\ \mu_{t,h+1}\neq 0}^{\tau}\prod_{q=1}^{h}\delta_{\mu_{t,q},0}\beta_{h+1,t}f_t=
\sum\limits_{t=1}^{\tau}\prod_{q=1}^{h+1}\delta_{\mu_{t,q},0}\beta_{h+1,t}f_{t},$$
one can assume that
$$[u,y_j]=0,\ 1\le j\leq h, \quad  [u,y_{h+1}]=\sum\limits_{t=1}^{\tau}\prod_{q=1}^{h+1}\delta_{\mu_{t,q},0}\beta_{h+1,t}f_{t}.$$

The conditions on $\mu_{t,j}$ and the following chain of equalities with $h+2\leq j\leq \eta$:
$$0=[u,[y_{h+1},y_j]]=[[u,y_{h+1}],y_j]-[[u,y_j],y_{h+1}]=
[\sum\limits_{t=1}^{\tau}\prod_{q=1}^{h+1}\delta_{\mu_{t,q},0}\beta_{h+1,t}f_{t},y_j]$$
$$-[\sum\limits_{t=1}^{\tau}\prod_{q=1}^{h}\delta_{\mu_{t,q},0}\beta_{j,t}f_{t},y_{h+1}]=
\sum\limits_{t=1}^{\tau}\prod_{q=1}^{h+1}\delta_{\mu_{t,q},0}\beta_{h+1,t}\mu_{t,j}f_{t}-
\sum\limits_{t=1}^{\tau}\prod_{q=1}^{h}\delta_{\mu_{t,q},0}\beta_{j,t}\mu_{t,h+1}f_{t}$$
imply
$$\prod_{q=1}^{h+1}\delta_{\mu_{t,q},0}\beta_{h+1,t}=0,\ \ \prod_{q=1}^{h}\delta_{\mu_{t,q},0}\beta_{j,t}\mu_{t,h+1}=0,\ \ h+2\leq j\leq \eta,\ 1\leq t\leq \tau.$$

Hence, we get
$$[u,y_{h+1}]=0, \quad  [u,y_j]=\sum\limits_{t=1}^{\tau}\prod_{q=1}^{h+1}\delta_{\mu_{t,q},0}\beta_{j,t}f_{t},\ h+2\leq j\leq \eta.$$

Thus, we obtain the products \eqref{eq34} for $h+1$. Taking the last possible value of $h=\eta$ in \eqref{eq34} we complete the proof.
\end{proof}

\begin{lem}\label{lem6} Let $\{e_1, \dots, e_{s}, x_1, \dots, x_{s}, f_1, \dots, f_{\tau}\}$
be elements of a basis of a solvable Lie algebra with the products:
$$\left\{\begin{array}{lll}
[e_i,x_i]=e_i+\sum\limits_{t=1}^{\tau}\beta_{i,i}^tf_{t},& 1\leq i\leq s,\\[1mm]
[e_i,x_j]=\sum\limits_{t=1}^{\tau}\beta_{i,j}^tf_{t},& 1\leq i\neq j\leq s,\\[1mm]
[f_t,x_j]=\mu_{t,j}f_t,& 1\leq j\leq s,\ 1\leq t\leq \tau,\\[1mm]
[e_i,[x_j,x_k]]=0, & 1\leq i,j,k \leq s,\\[1mm]
\end{array}\right.$$
where the structure constants $\mu_{t,j}$ satisfy the same conditions as in Lemma \ref{lem5}.

Then
$$[e_i,x_i]=e_i,\ 1\leq i\leq s,\quad [e_i,x_j]=0,\ 1\leq i\neq j\leq s.$$
\end{lem}

\begin{proof} Taking the change
$$e_i'=e_i+\sum\limits_{t=1,\mu_{t,i}\neq1}^{\tau}\frac{\beta_{i,i}^t}{1-\mu_{t,i}}f_t,\ 1\leq i\leq s,$$
we get
$$\begin{array}{ll}
[e_i',x_i]&=e_i+\sum\limits_{t=1}^{\tau}\beta_{i,i}^tf_t+\sum\limits_{t=1,\mu_{t,i}\neq1}^{\tau}\frac{\beta_{i,i}^t}{1-\mu_{t,i}}\mu_{t,i}f_t\\[1mm]
&=e_i'-\sum\limits_{t=1,\mu_{t,i}\neq1}^{\tau}\frac{\beta_{i,i}^t}{1-\mu_{t,i}}f_t+\sum\limits_{t=1}^{\tau}\beta_{i,i}^tf_t+
\sum\limits_{t=1,\mu_{t,i}\neq1}^{\tau}\frac{\beta_{i,i}^t}{1-\mu_{t,i}}\mu_{t,i}f_t\\[1mm]
&=e_i'+\sum\limits_{t=1}^{\tau}\beta_{i,i}^tf_t+\sum\limits_{t=1,\mu_{t,i}\neq1}^{\tau}\frac{\beta_{i,i}^t}{1-\mu_{t,i}}(\mu_{t,i}-1)f_t\\[1mm]
&=e_i'+\sum\limits_{t=1}^{\tau}\beta_{i,i}^tf_t-\sum\limits_{t=1,\mu_{t,i}\neq1}^{\tau}\beta_{i,i}^tf_t\\[1mm]
&=e_i'+\sum_{t=1}^{\tau}\delta_{\mu_{t,i},1}\beta_{i,i}^tf_t.
\end{array}$$
Therefore, we can assume
$$[e_i,x_i]=e_i+\sum_{t=1}^{\tau}\delta_{\mu_{t,i},1}\beta_{i,i}^tf_t,\quad 1\leq i\leq s.$$

The chain of the equalities with $1\leq i\neq j \leq s$ below
$$0=[e_i,[x_i,x_j]]=[[e_i,x_i],x_j]-[[e_i,x_j],x_i]=
[e_i+\sum_{t=1}^{\tau}\delta_{\mu_{t,i},1}\beta_{i,i}^tf_t,x_j]-
[\sum\limits_{t=1}^{\tau}\beta_{i,j}^tf_{t},x_i]$$
$$=\sum\limits_{t=1}^{\tau}\beta_{i,j}^tf_{t}+\sum_{t=1}^{\tau}\delta_{\mu_{t,i},1}\beta_{i,i}^t\mu_{t,j}f_t-
\sum\limits_{t=1}^{\tau}\beta_{i,j}^t\mu_{t,i}f_{t}=\sum\limits_{t=1}^{\tau}\beta_{i,j}^t(1-\mu_{t,i})f_{t}+\sum_{t=1}^{\tau}\delta_{\mu_{t,i},1}\beta_{i,i}^t\mu_{t,j}f_t,$$
implies
$$\beta_{i,j}^t(1-\mu_{t,i})=0,\quad \delta_{\mu_{t,i},1}\mu_{t,j}\beta_{i,i}^t=0,\ 1\leq t\leq \tau, \ 1\leq i\neq j \leq s.$$

Taking into account the condition on the coefficients $\mu_{t,j}$ we derive
$$\beta_{i,j}^t(1-\mu_{t,i})=0,\quad \delta_{\mu_{t,i},1}\beta_{i,i}^t=0,\ 1\leq t\leq \tau, \quad 1\leq i\neq j \leq s,$$
that means
$$[e_i,x_i]=e_i,\quad 1\leq i\leq s, \quad [e_i,x_j]=\sum\limits_{t=1}^{\tau}\delta_{\mu_{t,i},1}\beta_{i,j}^tf_t,\quad 1\leq i\neq j \leq s.$$

Making the change
$$e_i'=e_i-\sum\limits_{t=1,\ \mu_{t,1}\neq0}^{\tau}\frac{\delta_{\mu_{t,i},1}\beta_{i,1}^t}{\mu_{t,1}}f_t, \quad 2\leq i\leq s$$
one can assume
$$[e_i,x_i]=e_i, \quad [e_i,x_1]=\sum\limits_{t=1}^{\tau}\delta_{\mu_{t,1},0}\delta_{\mu_{t,i},1}\beta_{i,1}^tf_t, \quad 2\leq i\leq s.$$

From the equalities with $2\leq i\neq j\leq s$:
$$0=[e_i,[x_1,x_j]]=[[e_i,x_1],x_j]-[[e_i,x_j],x_1]=
\sum\limits_{t=1}^{\tau}\delta_{\mu_{t,1},0}\delta_{\mu_{t,i},1}\mu_{t,j}\beta_{i,1}^tf_t-
\sum\limits_{t=1}^{\tau}\delta_{\mu_{t,i},1}\mu_{t,1}\beta_{i,j}^tf_t$$
we obtain
$$\delta_{\mu_{t,1},0}\delta_{\mu_{t,i},1}\mu_{t,j}\beta_{i,1}^t=0, \quad \delta_{\mu_{t,i},1}\mu_{t,1}\beta_{i,j}^t=0,\ 1\leq t\leq \tau.$$
These imply
$$\delta_{\mu_{t,1},0}\delta_{\mu_{t,i},1}\beta_{i,1}^t=0, \quad \delta_{\mu_{t,i},1}\mu_{t,1}\beta_{i,j}^t=0,\ 1\leq t\leq \tau,$$
therefore,
$$[e_i,x_1]=0,\quad 2\leq i \leq s, \quad [e_i,x_j]=\sum\limits_{t=1}^{\tau}\delta_{\mu_{t,1},0}\delta_{\mu_{t,i},1}\beta_{i,j}^tf_t,\quad 2\leq i \neq j\leq s.$$

Note that we have
$$[e_1,x_j]=\sum\limits_{t=1}^{\tau}\delta_{\mu_{t,i},1}\beta_{1,j}^tf_t,\quad 2\leq j \leq s.$$

Applying now Lemma \ref{lem5} with $u:=e_1, \ y_i:=x_{i+1}$ and $\eta=s-1$ we obtain $[e_1,x_j]=0, \quad 2\leq j \leq s.$

Thus, we get
\begin{equation}\label{eq35}
\left\{\begin{array}{lll}
[e_i,x_i]=e_i,&1\leq i\leq s,\\[1mm]
[e_i,x_1]=0,& [e_1,x_j]=0, & 2\leq i, j\leq s,\\[1mm]
[e_i,x_j]=\sum\limits_{t=1}^{\tau}\delta_{\mu_{t,1},0}\delta_{\mu_{t,i},1}\beta_{i,j}^tf_t, & 2\leq i \neq j\leq s.\\[1mm]
\end{array}\right.
\end{equation}

The following also we prove by induction on $h$.
\begin{equation}\label{eq36}
\left\{\begin{array}{lll}
[e_i,x_i]=e_i,& 1\leq i\leq s,\\[1mm]
[e_i,x_j]=0,& 1\leq i\leq s, & 1\leq j\leq h,  i\neq j,\\[1mm]
[e_i,x_j]=0,& 1\leq j\leq s, & 1\leq i\leq h,  i\neq j,\\[1mm]
[e_i,x_j]=\sum\limits_{t=1}^{\tau}\delta_{\mu_{t,i},1}\delta_{\mu_{t,1},0}\beta_{i,j}^tf_t, & 2\leq i \neq j\leq s,\\[1mm]
[e_i,x_j]=\sum\limits_{t=1}^{\tau}\delta_{\mu_{t,i},1}\prod_{q=1}^{h}\delta_{\mu_{t,q},0}\beta_{i,j}^tf_t,& h+1\leq i \neq j\leq s.\\[1mm]
\end{array}\right.
\end{equation}

%
%

From \eqref{eq35} we get the base of the induction. Let us prove \eqref{eq36} for $h+1$.

Consider the change
$$e_{i}'=e_{i}-\sum\limits_{t=1,\ \mu_{t,h+1}\neq0}^{\tau}\frac{\delta_{\mu_{t,i},1}\prod_{q=1}^{h}\delta_{\mu_{t,q},0}\beta_{i,h+1}^t}{\mu_{t,h+1}}f_t,\ h+1\leq i\leq s.$$

Then we derive
$$[e_i',x_i]=e_i', \quad [e_i',x_j]=0,\ 1\leq j\leq h, \quad h+1\leq i\leq s$$
$$[e_i',x_{h+1}]=\sum\limits_{t=1}^{\tau}\delta_{\mu_{t,i},1}\prod_{q=1}^{h+1}\delta_{\mu_{t,q},0}\beta_{i,h+1}^tf_t,\ h+2\leq i\leq s.$$

The conditions of the lemma on $\mu_{t,j}$ and the equalities with $i,j \ (h+2\leq i,j\leq s)$
$$0=[e_i,[x_{h+1},x_j]]=[[e_i,x_{h+1}],x_j]-[[e_i,x_j],x_{h+1}]$$
$$=\sum\limits_{t=1}^{\tau}\Big(\delta_{\mu_{t,i},1}\prod_{q=1}^{h+1}\delta_{\mu_{t,q},0}\beta_{i,h+1}^t\mu_{t,j}-
\delta_{\mu_{t,i},1}\prod_{q=1}^{h}\delta_{\mu_{t,q},0}\beta_{i,j}^t\mu_{t,h+1}\Big)f_t$$
imply
$$\delta_{\mu_{t,i},1}\prod_{q=1}^{h+1}\delta_{\mu_{t,q},0}\beta_{i,h+1}^t=0,\quad
\delta_{\mu_{t,i},1}\prod_{q=1}^{h}\delta_{\mu_{t,q},0}\beta_{i,j}^t\mu_{t,h+1}=0,\ 1\leq t\leq \tau.$$
Therefore we get
$$\left\{\begin{array}{llll}
[e_i,x_{h+1}]=0, & 1\leq i\leq h,\\[1mm]
[e_{h+1},x_j]=\sum\limits_{t=1}^{\tau}\delta_{\mu_{t,h+1},1}
\prod_{q=1}^{h}\delta_{\mu_{t,q},0}\beta_{h+1,j}^tf_t, & h+2\leq j\leq s,\\[1mm]
[e_i,x_j]=\sum\limits_{t=1}^{\tau}\delta_{\mu_{t,i},1}
\prod_{q=1}^{h+1}\delta_{\mu_{t,q},0}\beta_{i,j}^tf_t, & h+2\leq i \neq j\leq s.\\[1mm]
\end{array}\right.$$

Applying Lemma \ref{lem5} with $u:=e_{h+1}, \ y_i:=x_i,\ h+2\leq i\leq s$ we derive $[e_{h+1},x_j]=0,\quad  h+2\leq j\leq s,$ which completes the proof of \eqref{eq36}.
%

Putting now $h=s$ in \eqref{eq36} we get the proof of the lemma.
\end{proof}

Let
$$[e_i,x_j]=\sum\limits_{t=i}^{n}\alpha_{i,j}^te_t,\quad 1\leq i \leq k, \quad 1\leq j \leq s.$$

Note that

\begin{equation}\label{eq2}
det\left(\begin{array}{cccccc}
\alpha_{1,1}^1&\alpha_{1,2}^1&\dots&\alpha_{1,s}^1\\[1mm]
\alpha_{2,1}^2&\alpha_{2,2}^2&\dots&\alpha_{2,s}^2\\[1mm]
\vdots&\vdots&\ddots&\vdots\\[1mm]
\alpha_{s,1}^s&\alpha_{s,2}^s&\dots&\alpha_{s,s}^s\\[1mm]
\end{array}\right)\neq 0.
\end{equation}

Indeed, if the determinant would be zero then ${ad_{x_i}}_{|N}$ would be nil-dependent which is not the case.

Thanks to \eqref{eq2} we obtain the system of equations:
$$\sum\limits_{j=1}^{s}A_{i,j}\alpha_{i,j}^i=1,\quad \sum\limits_{p=1}^{s}A_{i,p}\alpha_{j,p}^j=0, \quad 1\leq j\neq i\leq s,$$
that has a unique non trivial solution with respect to unknowns $A_{i,j}, \ 1\leq i,j \leq s.$

Let us take the change of basis elements $x_i, \ 1\leq i\leq s$ as follows
$x_i'=\sum_{j=1}^{s}A_{i,j}x_j, \ 1\leq i \leq s.$

Then we get
$$[e_i,x_i']=\sum\limits_{t=i}^{n}(\sum\limits_{j=1}^{s}A_{i,j}\alpha_{i,j}^t)e_t, \quad [e_j,x_i']=\sum\limits_{t=j}^{n}(\sum\limits_{p=1}^{s}A_{i,p}\alpha_{j,p}^t)e_t.$$

Thus, one can assume that
$$[e_i,x_i]=e_i+\sum\limits_{t=i+1}^{n}\alpha_{i,i}^te_t,\quad 1\leq i\leq s, \quad [e_i,x_j]=\sum\limits_{t=i+1}^{n}\alpha_{i,j}^te_t,\quad 1\leq i\neq j \leq s.$$

Consider sequentially the changes of basis elements $e_i$ for $i=1, 2, \dots, s$ as follows
$$e_i'=e_i+\sum_{t=i+1}^{s}C_{t}e_{t} \quad \mbox{with} \quad C_{i+1}=\alpha_{i,i}^{i+1},\ \quad C_t=\alpha_{i,i}^t+\sum\limits_{j=i+1}^{t-1}C_j\alpha_{j,i}^t,\quad i+2\leq t\leq s.$$

Then
$$\begin{array}{llll}
[e_i',x_i]&=e_i+\sum\limits_{t=i+1}^{n}\alpha_{i,i}^te_t+\sum\limits_{t=i+1}^{s}C_t[e_t,x_i]\\[1mm]
&=e_i+\sum\limits_{t=i+1}^{n}\alpha_{i,i}^te_t+\sum\limits_{t=i+1}^{s}C_t\sum\limits_{p=t+1}^{n}\alpha_{t,i}^pe_p\\[1mm]
& =e_i+\sum\limits_{t=i+1}^{s}\alpha_{i,i}^te_t+\sum\limits_{t=s+1}^{n}\alpha_{i,i}^te_t+
\sum\limits_{t=i+1}^{s}\Big(\sum\limits_{p=t+1}^{s}C_t\alpha_{t,i}^pe_p+\sum\limits_{p=s+1}^{n}C_t\alpha_{t,i}^pe_p\Big)e_i\\[1mm]
&+\sum\limits_{t=i+1}^{s}\alpha_{i,i}^te_t+\sum\limits_{t=s+1}^{n}\alpha_{i,i}^te_t+
\sum\limits_{t=i+2}^{s}\sum\limits_{p=i+1}^{t-1}C_p\alpha_{p,i}^te_t+\sum\limits_{t=s+1}^{n}\sum\limits_{p=i+1}^{s}C_p\alpha_{p,i}^te_t\\[1mm]
&=e_i+\alpha_{i,i}^{i+1}e_{i+1}+\sum\limits_{t=i+2}^{s}(\alpha_{i,i}^t+\sum\limits_{p=i+1}^{t-1}C_p\alpha_{p,i}^t)e_t+
\sum\limits_{t=s+1}^{n}(\alpha_{i,i}^te_t+\sum\limits_{p=i+1}^{s}C_p\alpha_{p,i}^t)e_t\\[1mm]
& =e_i'+\sum_{t=s+1}^{n}(\alpha_{i,i}^te_t+\sum\limits_{p=i+1}^{s}C_p\alpha_{p,i}^t)e_t.
\end{array}$$

Therefore, one can assume that
$$[e_i,x_i]=e_i+\sum\limits_{t=s+1}^{n}\alpha_{i,i}^te_t, \quad 1\leq i\leq s.$$

Consider now the products in $R$ modulo $N^2$. Then
\begin{equation}\label{modN2}
\left\{\begin{array}{llll}
[e_i,x_i]\equiv e_i+\sum\limits_{t=s+1}^{k}\alpha_{i,i}^te_t & (mod \ N^2), & 1\leq i\leq s, \\[1mm]
[e_i,x_j]\equiv \sum\limits_{t=i+1}^{k}\alpha_{i,j}^te_t & (mod \ N^2), & 1\leq i\neq j \leq s, \\[1mm]
[e_i,x_j]\equiv \alpha_{i,j}e_i+\sum\limits_{t=i+1}^{k}\alpha_{i,j}^te_t & (mod \ N^2), & s+1\leq i\leq k, \ 1\leq j \leq s.$$
\end{array}\right.
\end{equation}

The next lemma gives the explicit expressions for the products $[e_i,x_j]$ modulo $N^2$.

\begin{lem}\label{lem7} The following congruence equalities take place
\begin{equation}\label{eqlem4}
\left\{\begin{array}{lll}
[e_i,x_i]\equiv e_i & (mod \ N^2), & 1\leq i\leq s,\\[1mm]
[e_i,x_j]\equiv \alpha_{i,j}e_i & (mod \ N^2), & s+1\leq i\leq k, \  1\leq j \leq s, \\[1mm]
[e_i,x_j]\equiv 0 & (mod \ N^2),&1\leq i\neq j\leq s.\\[1mm]
\end{array}\right.
\end{equation}
\end{lem}
\begin{proof}
From the chain of congruences modulo $N^2$ for $1\leq i\neq j\leq s$ below
$$\begin{array}{ll}
0&\equiv[e_i,[x_i,x_j]]\equiv[[e_i,x_i],x_j]-[[e_i,x_j],x_i]\\[1mm]
&\equiv \sum\limits_{t=i+1}^{k}\alpha_{i,j}^te_t+\sum\limits_{t=s+1}^{k}\alpha_{i,i}^t(\alpha_{t,j}e_t+\sum\limits_{p=t+1}^{k}\alpha_{t,j}^pe_p)
-\sum\limits_{t=i+1}^{s}\alpha_{i,j}^t(\sum\limits_{p=t+1}^{k}\alpha_{t,i}^pe_p)\\[1mm]
&-\sum\limits_{t=s+1}^{k}\alpha_{i,j}^t(\alpha_{t,i}e_t+\sum\limits_{p=t+1}^{k}\alpha_{t,i}^pe_p)\ (mod \ N^2)\\[1mm]
&\equiv \sum\limits_{t=i+1}^{s}\alpha_{i,j}^te_t+\sum\limits_{t=s+1}^{k}\alpha_{i,j}^te_t
+\sum\limits_{t=s+1}^{k}\alpha_{i,i}^t(\alpha_{t,j}e_t+
\sum\limits_{p=t+1}^{k}\alpha_{t,j}^pe_p)-\sum\limits_{t=i+1}^{s}\alpha_{i,j}^t\sum\limits_{p=t+1}^{k}
\alpha_{t,i}^pe_p\\[1mm]
&-\sum\limits_{t=s+1}^{k}\alpha_{i,j}^t(\alpha_{t,i}e_t+\sum\limits_{p=t+1}^{k}\alpha_{t,i}^pe_p) \quad (mod \ N^2),
\end{array}$$
we derive $\alpha_{i,j}^t=0$ for any $1\leq i\neq j \leq s, \ i+1\leq t \leq s$. Hence, we get
\begin{equation}\label{eixjineqj}
[e_i,x_j]\equiv \sum\limits_{t=s+1}^{k}\alpha_{i,j}^te_t\  (mod \ N^2),\quad 1\leq i\neq j \leq s.
\end{equation}

Consider sequentially the changes of basis elements $e_i$ for $i=s, \dots, k$ as follows
$$e_i'=e_i+\sum_{t=i+1}^{k}C_{t}e_{t},$$
where
$$C_{t}=\left\{\begin{array}{ll}
0,&\mbox{if} \quad \alpha_{t,1}=\alpha_{i,1},\\[1mm]
\frac{\alpha_{i,1}^t+\sum\limits_{p=i+1}^{t-1}C_{p}\alpha_{p,1}^t}{\alpha_{i,1}-\alpha_{t,1}},&\mbox{if} \quad \alpha_{t,1}\neq \alpha_{i,1}.\\[1mm]
\end{array}\right. $$

Then
$$\begin{array}{lll}
[e_i',x_1]&\equiv\alpha_{i,1}e_i+\sum\limits_{t=i+1}^{k}\alpha_{i,1}^te_t+
\sum\limits_{t=i+1}^{k}C_{t}(\alpha_{t,1}e_t+\sum\limits_{p=t+1}^{k}\alpha_{t,1}^pe_p)\\[1mm]
&\equiv \alpha_{i,1}e_i+\sum\limits_{t=i+1}^{k}(\alpha_{i,1}^t+C_{t}\alpha_{t,1})e_t+
\sum\limits_{p=i+2}^{k}\sum\limits_{t=i+1}^{p-1}C_{t}\alpha_{t,1}^pe_p\\[1mm]
&\equiv \alpha_{i,1}e_i+\sum\limits_{t=i+1}^{k}(\alpha_{i,1}^t+C_{t}\alpha_{t,1})e_t+
\sum\limits_{t=i+2}^{k}\sum\limits_{p=i+1}^{t-1}C_{p}\alpha_{p,1}^te_t\\[1mm]
&\equiv\alpha_{i,1}e_i'+(\alpha_{i,1}^{i+1}+C_{i+1}(\alpha_{i+1,1}-\alpha_{i,1}))e_{i+1}\\[1mm]
&+\sum\limits_{t=i+2}^{k}(\alpha_{i,1}^t+C_{t}(\alpha_{t,1}-\alpha_{i,1})+
\sum\limits_{p=i+1}^{t-1}C_{p}\alpha_{p,1}^t)e_t \  (mod \ N^2).
\end{array}$$

Thus, we can suppose that
\begin{equation}\label{eix1}
[e_i,x_1]\equiv\alpha_{i,1}e_i+\sum_{t=i+1}^{k}\delta_{\alpha_{t,1}, \alpha_{i,1}}\alpha_{i,1}^te_t \  (mod \ N^2),\quad s+1\leq i\leq k,
\end{equation}
where $\delta_{-, -}$ is the Kronecker delta.

For $s+1\leq i\leq k$ and $2\leq j\leq s$ we consider the followings congruences

$$\begin{array}{lll}
0 &\equiv[e_i,[x_1,x_j]]\equiv[[e_i,x_1],x_j]-[[e_i,x_j],x_1]\\[1mm]
& \equiv\alpha_{i,1}(\alpha_{i,j}e_i+\sum\limits_{t=i+1}^{k}\alpha_{i,j}^te_t)+
\sum_{t=i+1}^{k}\delta_{\alpha_{t,1},\alpha_{i,1}}\alpha_{i,1}^t
(\alpha_{t,j}e_t+\sum\limits_{p=t+1}^{k}\alpha_{t,j}^pe_p)\\[1mm]
&-\alpha_{i,j}(\alpha_{i,1}e_i+\sum_{t=i+1}^{k}\delta_{\alpha_{t,1}, \alpha_{i,1}}\alpha_{i,1}^te_t)-
\sum\limits_{t=i+1}^{k}\alpha_{i,j}^t(\alpha_{t,1}e_t+\sum_{p=t+1}^{k}\delta_{\alpha_{p,1}, \alpha_{t,1}}\alpha_{t,1}^pe_p)\\[1mm]
&\equiv \sum\limits_{t=i+1}^{k}(\alpha_{i,1}-\alpha_{t,1})\alpha_{i,j}^te_t+
\sum_{t=i+1}^{k}\delta_{\alpha_{t,1},\alpha_{i,1}}\alpha_{i,1}^t(\alpha_{t,j}-\alpha_{i,j})e_t\\[1mm]
&-\sum_{t=i+1}^{k}\delta_{\alpha_{t,1},\alpha_{i,1}}\alpha_{i,1}^t(\sum\limits_{p=t+1}^{k}
\alpha_{t,j}^pe_p)-\sum\limits_{t=i+1}^{k}\alpha_{i,j}^t(\sum_{p=t+1}^{k}\delta_{\alpha_{p,1},\alpha_{t,1}}\alpha_{t,1}^pe_p) \\[1mm]
&\equiv \sum\limits_{t=i+1}^{k}(\alpha_{i,1}-\alpha_{t,1})\alpha_{i,j}^te_t+\sum_{t=i+1}^{k}\delta_{\alpha_{t,1},\alpha_{i,1}}\alpha_{i,1}^t(\alpha_{t,j}-\alpha_{i,j})e_t\\[1mm]
&+\sum\limits_{p=i+2}^{k}\Big(\sum_{t=i+1}^{p-1}(\delta_{\alpha_{t,1},\alpha_{i,1}}\alpha_{i,1}^t\alpha_{t,j}^p-\delta_{\alpha_{p,1},\alpha_{t,1}}\alpha_{i,j}^t\alpha_{t,1}^p)\Big)e_p \  (mod \ N^2).
\end{array}$$

These congruences imply the equalities
$$(\alpha_{i,1}-\alpha_{t,1})\alpha_{i,j}^t=0, \quad \delta_{\alpha_{t,1},\alpha_{i,1}}\alpha_{i,1}^t(\alpha_{t,j}-\alpha_{i,j})=0.$$

Due to condition {\bf A)} and the products \eqref{modN2}, \eqref{eix1} we obtain
\begin{equation}\label{eq1234}
\left\{\begin{array}{lll}
[e_i,x_1]\equiv\alpha_{i,1}e_i & (mod \ N^2),& s+1\leq i\leq k,\\[1mm]
[e_i,x_j]\equiv\alpha_{i,j}e_i+\sum\limits_{t=i+1}^{k}\delta_{\alpha_{i,1},\alpha_{t,1}}\alpha_{i,j}^te_t & (mod \ N^2),&
s+1\leq i\leq k,\quad 2\leq j \leq s.
\end{array}\right.
\end{equation}

By induction on $h$ we prove the followings
\begin{equation}\label{eqHinduction}
\left\{\begin{array}{llll}
[e_i,x_j]\equiv \alpha_{i,j}e_i & (mod \ N^2), & s+1\leq i\leq k, & 1\leq j\leq h,\\[1mm]
[e_i,x_j]\equiv \alpha_{i,j}e_i+\sum\limits_{t=i+1}^{k}\prod_{q=1}^{h}\delta_{\alpha_{i,q},\alpha_{t,q}}\alpha_{i,j}^te_t & (mod \ N^2),&
 s+1\leq i\leq k,& h+1\leq j \leq s.
\end{array}\right.
\end{equation}

The base of the induction (that is for $h=1$) is clear thanks to \eqref{eq1234}. Assuming that \eqref{eqHinduction} is true for any $h$ we prove for $h+1$.

Consider the change
$$e_i'=e_i+\sum\limits_{t=i+1,\ \alpha_{i,h+1}\neq \alpha_{t,h+1}}^{k}
\prod_{q=1}^{h}\delta_{\alpha_{i,q},\alpha_{t,q}}A_{i,t}e_t,\ s+1\leq i\leq k,$$
with
$$A_{i,t}=\frac{\alpha_{i,h+1}^t+\sum\limits_{p=i+1,\alpha_{i,h+1}\neq \alpha_{t,h+1}}^{t-1}\prod_{q=1}^{h}\delta_{\alpha_{i,q},\alpha_{t,q}}A_{i,p}\alpha_{p,h+1}^t}
{\alpha_{i,h+1}-\alpha_{t,h+1}},\ i+1\leq t\leq s.$$

Then we derive
$$\begin{array}{lll}
[e_i',x_{h+1}]&=\alpha_{i,h+1}e_i+\sum\limits_{t=i+1}^{k}\prod_{q=1}^{h}\delta_{\alpha_{i,q},\alpha_{t,q}}\alpha_{i,h+1}^te_t\\[1mm]
&+\sum\limits_{t=i+1,\ \alpha_{i,h+1}\neq \alpha_{t,h+1}}^{k}\prod_{q=1}^{h}\delta_{\alpha_{i,q},\alpha_{t,q}}A_{i,t}
(\alpha_{t,h+1}e_t+\sum\limits_{p=t+1}^{k}\prod_{q=1}^{h}\delta_{\alpha_{i,q},\alpha_{p,q}}\alpha_{t,h+1}^pe_p)\\[1mm]
&-\alpha_{i,h+1}e_i+\sum\limits_{t=i+1}^{k}\prod_{q=1}^{h}\delta_{\alpha_{i,q},\alpha_{t,q}}\alpha_{i,h+1}^te_t+
\sum\limits_{t=i+1,\ \alpha_{i,h+1}\neq \alpha_{t,h+1}}^{k}\prod_{q=1}^{h}\delta_{\alpha_{i,q},\alpha_{t,q}}A_{i,t}\alpha_{t,h+1}e_t\\[1mm]
&+\sum\limits_{p=i+2}^{k}\sum\limits_{t=i+1,\ \alpha_{i,h+1}\neq \alpha_{t,h+1}}^{p-1}\prod_{q=1}^{h}\delta_{\alpha_{i,q},\alpha_{t,q}}
\delta_{\alpha_{i,q},\alpha_{p,q}}A_{i,t}\alpha_{t,h+1}^pe_p\\[1mm]
&=\alpha_{i,h+1}e_i+\sum\limits_{t=i+1}^{k}\prod_{q=1}^{h}\delta_{\alpha_{i,q},\alpha_{t,q}}\alpha_{i,h+1}^te_t+
\sum\limits_{t=i+1,\ \alpha_{i,h+1}\neq \alpha_{t,h+1}}^{k}\prod_{q=1}^{h}\delta_{\alpha_{i,q},\alpha_{t,q}}A_{i,t}\alpha_{t,h+1}e_t\\[1mm]
&+\sum\limits_{t=i+2}^{k}\sum\limits_{p=i+1,\ \alpha_{i,h+1}\neq \alpha_{p,h+1}}^{t-1}
\prod_{q=1}^{h}\delta_{\alpha_{i,q},\alpha_{p,q}}\delta_{\alpha_{i,q},\alpha_{t,q}}A_{i,p}\alpha_{p,h+1}^te_t\\[1mm]
&=\alpha_{i,h+1}e_i'+\sum\limits_{t=i+1}^{k}\prod_{q=1}^{h+1}\delta_{\alpha_{i,q},\alpha_{t,q}}A_{i,t}(\alpha_{i,h+1}-\alpha_{t,h+1})e_t\\[1mm]
&+\sum\limits_{t=i+1,\ \ \alpha_{i,h+1}\neq \alpha_{t,h+1}}^{k}\prod_{q=1}^{h}\delta_{\alpha_{i,q},\alpha_{t,q}}(A_{i,t}(\alpha_{t,h+1}-\alpha_{i,h+1})+
A_{i,t}(\alpha_{i,h+1}-\alpha_{t,h+1}))e_t\\[1mm]
&=\alpha_{i,h+1}e_i'+\sum\limits_{t=i+1}^{k}\prod_{q=1}^{h+1}\delta_{\alpha_{i,q},\alpha_{t,q}}
A_{i,t}(\alpha_{i,h+1}-\alpha_{t,h+1})e_t.
\end{array}$$

Consequently, one can assume that
$$[e_i,x_{h+1}] \equiv\alpha_{i,h+1}e_i+\sum\limits_{t=i+1}^{k}\prod_{q=1}^{h+1}\delta_{\alpha_{i,q},\alpha_{t,q}}
\alpha_{i,h+1}^te_t \quad (mod \ N^2).$$

The following congruences with $s+1\leq i\leq k,\ h+2\leq j\leq s$:
$$\begin{array}{lll}
0&\equiv [e_i,[x_j,x_{h+1}]]=[[e_i,x_j],x_{h+1}]-[[e_i,x_{h+1}],x_j]\\[1mm]
&\equiv\sum\limits_{t=i+1}^{k}\prod_{q=1}^{h}\delta_{\alpha_{i,q},\alpha_{t,q}}\Big
(\delta_{\alpha_{i,h+1},\alpha_{t,h+1}}(\alpha_{i,j}-\alpha_{t,j})\alpha_{i,h+1}^t+
(\alpha_{t,h+1}-\alpha_{i,h+1})\alpha_{i,j}^t\Big)e_t\\[1mm]
&+\sum\limits_{t=i+1}^{k}\sum\limits_{p=t+1}^{k}\prod_{q=1}^{h}\delta_{\alpha_{i,q},\alpha_{t,q}}
\prod_{q=1}^{h}\delta_{\alpha_{t,q},\alpha_{p,q}}\Big(\delta_{\alpha_{t,h+1},\alpha_{p,h+1}}\alpha_{i,j}^t\alpha_{t,h+1}^p-
\delta_{\alpha_{i,h+1},\alpha_{t,h+1}}\alpha_{i,h+1}^t\alpha_{t,j}^p\Big)e_p \ (mod \ N^2),
\end{array}$$
give
$$\prod_{q=1}^{h+1}\delta_{\alpha_{i,q},\alpha_{t,q}}\alpha_{i,h+1}^t=0,\ \quad
\prod_{q=1}^{h}\delta_{\alpha_{i,q},\alpha_{t,q}}(\alpha_{t,h+1}-\alpha_{i,h+1})\alpha_{i,j}^t=0,\ \ i+1\leq t\leq k.$$

Therefore, we get
$$[e_i,x_{h+1}]\equiv\alpha_{i,h+1}e_i \ (mod \ N^2),\quad $$
$$[e_i,x_j]\equiv\alpha_{i,j}e_i+\sum\limits_{t=i+1}^{k}
\prod_{q=1}^{h+1}\delta_{\alpha_{i,q},\alpha_{t,q}}\alpha_{i,j}^te_t \ (mod \ N^2),\quad s+1\leq i\leq k,\quad h+2\leq j \leq s.$$
Thus, \eqref{eqHinduction} is proved. Taking $h=s$ we get

\begin{equation}\label{eixj}
[e_i,x_j]\equiv\alpha_{i,j}e_i \  (mod \ N^2),\quad s+1\leq i\leq k,\quad 1\leq j \leq s.
\end{equation}

If in the assertion of Lemma \ref{lem6} we assume $f_i=e_{s+i}$, $\tau=k-s$ and $\mu_{t,j}=\alpha_{t,j}$, then we get the first product in \eqref{modN2} and the products \eqref{eixjineqj}, \eqref{eixj}. Therefore, applying Lemma \ref{lem6} we complete the proof.
%
%
%
\end{proof}

The next result gives the vanishing of the product $[Q,Q]$ modulo $N^2$.

\begin{lem}\label{lem8} One has
$$[x_i,x_j]\equiv 0 \ (mod \ N^2), \quad 1\leq i, j \leq s.$$
\end{lem}
\begin{proof} Let set
$$[x_i,x_j]\equiv\sum\limits_{t=1}^{k}\beta_{i,j}^te_t \ (mod \ N^2),\quad 1\leq i\neq j \leq s.$$

For $i,j,p$ such that $1\leq i,j,p\leq s,\ i\neq j,\ i\neq p,\ j\neq p$ we have
$$[x_p,[x_i,x_j]]\equiv[[x_p,x_i],x_j]-[[x_p,x_j],x_i]\equiv\beta_{p,i}^je_j+\sum\limits_{t=s+1}^{k}\beta_{p,i}^t\alpha_{t,j}e_t-
\beta_{p,j}^ie_i-\sum\limits_{t=s+1}^{k}\beta_{p,j}^t\alpha_{t,i}e_t\ (mod \ N^2).$$

On the other hand, it is clear that
$$[x_p,[x_i,x_j]]\equiv[x_p,\sum\limits_{t=1}^{k}\beta_{i,j}^te_t]\equiv-\beta_{i,j}^pe_p-\sum\limits_{t=s+1}^{k}\beta_{i,j}^t\alpha_{t,p}e_t\ (mod \ N^2).$$
Therefore,
$$\beta_{i,j}^p=0,\quad 1\leq i\neq j\neq p \leq s \Rightarrow
[x_i,x_j]\equiv\beta_{i,j}^ie_i+\beta_{i,j}^je_j+\sum\limits_{t=s+1}^{k}\beta_{i,j}^te_t \ (mod \ N^2), \quad 1\leq i \neq j \leq s.$$

Taking the change
$$x_i'=x_i-\sum\limits_{t=1}^{i-1}\beta_{i,t}^te_t-\sum\limits_{t=i+1}^{s}\beta_{i,t}^te_t,\ 1\leq i\leq s,$$
one can assume that
$$[x_i,x_j]\equiv\sum\limits_{t=s+1}^{k}\beta_{i,j}^te_t\ (mod \ N^2), \quad 1\leq i\neq j \leq s.$$

Applying induction by $h$ we find the expressions for the following products:
\begin{equation}\label{eq5}
\left\{\begin{array}{lll}
[x_i,x_j]\equiv0 & (mod \ N^2),&  1\leq i\leq h, \ 1\leq j\leq s,\\[1mm]
[x_i,x_j]\equiv\sum\limits_{t=s+1}^{k}\prod_{q=1}^{h}\delta_{\alpha_{t,q},0}\beta_{i,j}^te_t & (mod \ N^2), & h+1\leq i \neq j \leq s.\\[1mm]
\end{array}\right.
\end{equation}

For $h=1$ we have
$$[x_1,x_j]\equiv\sum\limits_{t=s+1}^{k}\beta_{1,j}^te_t \ (mod \ N^2), \quad 2\leq j \leq s.$$

It is easy to see that our products satisfy to the conditions of Lemma \ref{lem5} with
$u:=x_1, \ y_{i-1}:=x_i, \ \eta:=s+1, \ f_t:=e_t, \ \tau:=k-s, \ \mu_{t,j}:=\alpha_{t,j}$. Therefore, we obtain $[x_1,x_j]\equiv0, \quad 2\leq j \leq s.$

For $2\leq i<j\leq s$ we have
$$[x_1,[x_i,x_j]]=[[x_1,x_i],x_j]-[[x_1,x_j],x_i]\equiv0 \ (mod \ N^2).$$
On the other hand,
$$[x_1,[x_i,x_j]]\equiv[x_1,\sum\limits_{t=s+1}^{k}\beta_{i,j}^te_t]\equiv-\sum\limits_{t=s+1}^{k}\alpha_{t,1}\beta_{i,j}^te_t.$$
Consequently,
$$\alpha_{t,1}\beta_{i,j}^t=0,\ s+1\leq t\leq k,\ 2\leq i<j\leq s.$$
Hence,
$$[x_1,x_i]\equiv0\ (mod \ N^2),\ 2\leq i\leq s,\quad [x_i,x_j]\equiv\sum\limits_{t=s+1}^{k}\delta_{\alpha_{t,1},0}\beta_{i,j}^te_t\ (mod \ N^2), \ \  2\leq i< j \leq s.$$
Thus, \eqref{eq5} is true for $h=1$. Assuming now that \eqref{eq5} is true for any $h$ we prove it for $h+1$. Thanks to the induction assumption we have
$$[x_{h+1},x_{i}]\equiv\sum\limits_{t=s+1}^k\prod_{q=1}^{h}\delta_{\alpha_{t,q},0}\beta_{h+1,i}^te_t\ (mod \ N^2),\  h+2\leq i\leq s.$$

Assuming $u:=x_{h+1}, \ y_{i-h-1}:=x_i, \ \eta:=s-h-1, \ f_{t-s}=e_t,\ s+1\leq t\leq k,\ \tau=k-s$ with $h+2\leq i\leq s$ we are in the conditions of Lemma \ref{lem5}. Therefore, we derive
$$[x_{h+1},x_{i}]\equiv0\ (mod \ N^2),\  h+2\leq i\leq s.$$

The equalities with $h+2\leq j< i \leq s$
$$0=[x_{i},[x_j,x_{h+1}]]=[[x_i,x_j],x_{h+1}]-[[x_i,x_{h+1}],x_j]\equiv
\sum\limits_{t=s+1}^{k}\prod_{q=1}^{h}\delta_{\alpha_{t,q},0}\alpha_{t,h+1}\beta_{i,j}^te_t \ (mod \ N^2),$$
imply
$$\prod_{q=1}^{h}\delta_{\alpha_{t,q},0}\alpha_{t,h+1}\beta_{i,j}^t=0,\quad s+1\leq t\leq k,\ \ h+2\leq j< i \leq s.$$
Therefore, we obtain
$$\left\{\begin{array}{lll}
[x_i,x_j]\equiv0 & (mod \ N^2),&  1\leq i\leq h+1, \ 1\leq j\leq s,\\[1mm]
[x_i,x_j]\equiv\sum\limits_{t=s+1}^{k}\prod\limits_{q=1}^{h+1}\delta_{\alpha_{t,q},0}\beta_{i,j}^te_t & (mod \ N^2), & h+2\leq i < j\leq s,\\[1mm]
\end{array}\right.$$
which proves \eqref{eq5}. Putting $h=s$ we complete the proof of the lemma.
\end{proof}
Combining the results of Lemmas \eqref{lem7} and \eqref{lem8} we get the following corollary.
\begin{cor} Modulo $N^2$ the following take place
\begin{equation}\label{eq6}
\left\{\begin{array}{lll}
[e_i,x_i]\equiv e_i & (mod \ N^2), & 1\leq i\leq s,\\[1mm]
[e_i,x_j]\equiv \alpha_{i,j}e_i & (mod \ N^2), & s+1\leq i\leq k,\ \ 1\leq j \leq s,\\[1mm]
[x_i,x_j]\equiv 0 & (mod \ N^2), &  1\leq j< i \leq s.\\[1mm]
\end{array}\right.
\end{equation}
\end{cor}

Now we give the description of maximal solvable extensions of a finite-dimensional non-split pure non-characteristically nilpotent Lie algebra $\mathcal{N}$.

\begin{thm} \label{mainthm} Let $\mathcal{R}$ be a complex maximal solvable extension of $n$-dimensional non-split nilpotent (non-characteristically nilpotent) Lie algebra $\mathcal{N}$ which satisfies the condition ${\bf A)}$. Then $\mathcal{R}$ admits a basis $\{e_1, \dots, e_s, \dots, e_{k}, \dots, e_{n}, x_1, \dots, x_s\}$ such that  the table of multiplications of $\mathcal{R}$ in this basis has the following form:
$$\left\{\begin{array}{lll}\label{eq9}
[e_i,e_j]=\sum\limits_{t=k+1}^{n}c_{i,j}^te_t,& 1\leq i\neq j\leq n,\\[1mm]
[e_i,x_j]=\alpha_{i,j}e_i,& 1\leq i\leq n,\ \ 1\leq j \leq s,\\[1mm]
\end{array}\right.$$
where $codim\mathcal{N}=s$ and $\alpha_{i,j}$ are coordinates of the basic vectors in the fundamental solutions to the system $S_e$.
\end{thm}
\begin{proof} First we prove by induction the validity of the following products modulo $N^m$ for any $m$ less or equal to the nilindex of $N$.

\begin{equation}\label{eqmain}
\left\{\begin{array}{lll}
[e_i,x_i]\equiv e_i & (mod \ N^2), & 1\leq i\leq s, \\[1mm]
[e_i,x_j]\equiv \alpha_{i,j}e_i & (mod \ N^2),& s+1\leq i\leq k,\   1\leq j \leq s,\\[1mm]
[x_i,x_j]\equiv 0 & (mod \ N^2),  & 1\leq j< i \leq s.\\[1mm]
\end{array}\right.
\end{equation}
The base of the induction holds due to \eqref{eq6}. Let us suppose that \eqref{eqmain} is true for $m$. Then we have
$$\left\{\begin{array}{lll}
[e_i,x_i]\equiv e_i+\sum\limits_{t=1}^{s_m}\alpha_{i,i}^tf_t & (mod \ N^{m+1}), & 1\leq i\leq s,\\[1mm]
[e_i,x_j]\equiv \sum\limits_{t=1}^{s_m}\alpha_{i,j}^tf_t & (mod \ N^{m+1}), & 1\leq i\neq j\leq s,\\[1mm]
[e_i,x_j]\equiv \alpha_{i,j}e_i+\sum\limits_{t=1}^{s_m}\alpha_{i,j}^tf_t & (mod \ N^{m+1}), & s+1\leq i\leq k,\ \ 1\leq j \leq s,\\[1mm]
[x_i,x_j]\equiv \sum\limits_{t=1}^{s_m}\beta_{i,j}^tf_t & (mod \ N^{m+1}), &  1\leq j< i \leq s,\\[1mm]
\end{array}\right.
$$
where $f_1,\dots,f_{s_m}$ is a basis of $N$ which lies in $N^{m}\setminus N^{m+1}$.

Without loss of generality on can take
$$f_i=[[\ldots[[e_{i_1},e_{i_2}],e_{i_3}],\ldots,e_{i_{m-1}}],e_{i_m}],\quad 1\leq i\leq s_m,\ 1\leq i_t\leq k,\ 1\leq t\leq m.$$

For $i,j$ ($1\leq i\leq s_m,\ 1\leq j\leq s$) we have
$$\begin{array}{ll}
[f_i,x_j]&=[[[\ldots[[e_{i_1},e_{i_2}],e_{i_{3}}],\ldots],e_{i_m}],x_j]\\[1mm]
&=\mbox{applying} \ (m-1)-\mbox{times the Jacoby identity one gets}\\[1mm]
&\equiv(\alpha_{i_{m},j}+\alpha_{i_{m-1},j}+\ldots+\alpha_{i_{2},j}+\alpha_{i_{1},j})[[\ldots[[e_{i_1},e_{i_2}],e_{i_{3}}],\ldots],e_{i_m}] \\[1mm]
&\equiv(\alpha_{i_{m},j}+\alpha_{i_{m-1},j}+\ldots+\alpha_{i_{2},j}+\alpha_{i_{1},j})f_i \quad (mod \ N^m).
\end{array}$$
Therefore,
$$[f_i,x_j]\equiv\lambda_{i,j}f_i \quad (mod \ N^m), \ 1\leq i\leq s_m,\ 1\leq j\leq s.$$
Note here that for any $i,j$ ($1\leq i\neq j \leq s_m)$ there exists $t$ ($1\leq t\leq s$) such that $\lambda_{i,t}\neq \lambda_{j,t}.$

Now we write out the obtained products
\begin{equation}\label{eq12345}
\left\{\begin{array}{lll}
[e_i,x_i]\equiv e_i+\sum\limits_{t=1}^{s_m}\alpha_{i,i}^tf_t & (mod \ N^{m+1}), & 1\leq i\leq s,\\[1mm]
[e_i,x_j]\equiv \sum\limits_{t=1}^{s_m}\alpha_{i,j}^tf_t & (mod \ N^{m+1}), & 1\leq i\neq j\leq s,\\[1mm]
[f_i,x_j]\equiv \lambda_{i,j}f_i & (mod \ N^{m+1}), & 1\leq i\leq s_m,\ 1\leq j\leq s.
\end{array}\right.
\end{equation}

Assuming $s_m:=\tau, \ \lambda_{t,j}:=\mu_{t,j}, \ \alpha_{i,j}^t:=\beta_{i,j}^t$ in the products \eqref{eq12345} we note that products satisfy the conditions of Lemma \ref{lem6}. Therefore, applying  Lemma \ref{lem6} we derive
$$[e_i,x_i]=e_i,\quad [e_i,x_j]=0 \quad (mod \ N^{m+1}),\quad 1\leq i\neq j\leq s.$$

Making the change as follows
$$e_i'=e_i+\sum\limits_{t=1, \lambda_{t,1}\neq \alpha_{i,1}}^{s_m}\frac{\alpha_{i,1}^t}{\alpha_{i,1}-\lambda_{t,1}}f_t,\quad s+1\leq i\leq k,$$
we obtain
$$[e_i',x_1]\equiv\alpha_{i,1}e_i'-\sum\limits_{t=1, \lambda_{t,1}\neq \alpha_{i,1}}^{s_m}\alpha_{i,1}^tf_t+\sum\limits_{t=1}^{s_m}\alpha_{i,1}^tf_t
=\alpha_{i,1}e_i'+\sum\limits_{t=1}^{s_m}\delta_{\alpha_{i,1},\lambda_{t,1}}\alpha_{i,1}^tf_t.$$

Hence,
$$[e_i,x_1]\equiv\alpha_{i,1}e_i+
\sum\limits_{t=1}^{s_m}\delta_{\alpha_{i,1},\lambda_{t,1}}\alpha_{i,1}^tf_t \quad (mod \ N^{m+1}), \quad s+1\leq i\leq k.$$

From the following equalities with $s+1\leq i\leq k,\ 2\leq j\leq s$
$$\begin{array}{ll}
0&\equiv[e_i,[x_1,x_j]]=[[e_i,x_1],x_j]-[[e_i,x_j],x_1]\\[1mm]
&\equiv\alpha_{i,1}(\alpha_{i,j}e_i+\sum\limits_{t=1}^{s_m}\alpha_{i,j}^tf_t)+\sum\limits_{t=1}^{s_m}\delta_{\alpha_{i,1},\lambda_{t,1}}\lambda_{t,j}\alpha_{i,1}^tf_t\\[1mm]
&-\alpha_{i,j}(\alpha_{i,1}e_i+\sum\limits_{t=1}^{s_m}\delta_{\alpha_{i,1},\lambda_{t,1}}\alpha_{i,1}^tf_t)-
\sum\limits_{t=1}^{s_m}\lambda_{t,1}\alpha_{i,j}^tf_t\\[1mm]
&\equiv\sum\limits_{t=1}^{s_m}\Big((\alpha_{i,1}-\lambda_{t,1})\alpha_{i,j}^t+
\delta_{\alpha_{i,1},\lambda_{t,1}}(\lambda_{t,j}-\alpha_{i,j})\alpha_{i,1}^t\Big)f_t \ (mod \ N^{m+1}),
\end{array}$$
we derive
$$(\alpha_{i,1}-\lambda_{t,1})\alpha_{i,j}^t=0, \quad \delta_{\alpha_{i,1},\lambda_{t,1}}\alpha_{i,1}^t=0,\ 1\leq t\leq s_m.$$
These imply
\begin{equation}\label{eq8}
\left\{\begin{array}{lll}
[e_i,x_1]\equiv\alpha_{i,1}e_i & (mod \ N^{m+1}), & s+1\leq i\leq k,\\[1mm]
[e_i,x_j]\equiv\alpha_{i,j}e_i+\sum\limits_{t=1}^{s_m}\delta_{\alpha_{i,1},\lambda_{t,1}}\alpha_{i,j}^tf_t & (mod \ N^{m+1}),& s+1\leq i\leq k,\ \ 2\leq j \leq s.\\[1mm]
\end{array}\right.
\end{equation}

By induction on $h$ we prove the following:
\begin{equation}\label{eq144}
\left\{\begin{array}{llll}
[e_i,x_j]\equiv\alpha_{i,j}e_i & (mod \ N^{m+1}), & s+1\leq i\leq k, \ 1\leq j\leq h,\\[1mm]
[e_i,x_j]\equiv\alpha_{i,j}e_i+\sum\limits_{t=1}^{s_m}
\prod_{q=1}^{h}\delta_{\alpha_{i,q},\lambda_{t,q}}\alpha_{i,j}^tf_t & (mod \ N^{m+1}), & s+1\leq i\leq k, \ h+1\leq j \leq s.\\[1mm]
\end{array}\right.
\end{equation}

The base of the induction holds due to \eqref{eq8}. Let us prove \eqref{eq144} for $h+1$ assuming that it is true for $h$.

Taking the change
$$e_i'=e_i+\sum\limits_{t=1, \lambda_{t,h+1}\neq \alpha_{i,h+1}}^{s_m}\frac{
\prod_{q=1}^{h}\delta_{\alpha_{i,q},\lambda_{t,q}}\alpha_{i,h+1}^t}
{\alpha_{i,h+1}-\lambda_{t,h+1}}f_t,\quad s+1\leq i\leq k,$$
we obtain
$$[e_i',x_{j}]\equiv\alpha_{i,j}e_i+\sum\limits_{t=1, \lambda_{t,h+1}\neq \alpha_{i,h+1}}^{s_m}\frac{
\prod_{q=1}^{h}\delta_{\alpha_{i,q},\lambda_{t,q}}\alpha_{i,h+1}^t}
{\alpha_{i,h+1}-\lambda_{t,h+1}}\lambda_{t,j}f_t=\alpha_{i,j}e_i',\ 1\leq j\leq h,$$

$$[e_i',x_{h+1}]\equiv\alpha_{i,h+1}e_i'+\sum\limits_{t=1}^{s_m}\prod_{q=1}^{h+1}\delta_{\alpha_{i,q},\lambda_{t,q}}\alpha_{i,h+1}^tf_t \ (mod \ N^{m+1}).$$
Hence,
$$\left\{\begin{array}{llll}
[e_i,x_j]\equiv \alpha_{i,j}e_i & (mod \ N^{m+1}),  & s+1\leq i\leq k, \ 1\leq j\leq h ,\\[1mm]
[e_i,x_{h+1}]\equiv\alpha_{i,h+1}e_i+\sum\limits_{t=1}^{s_m}\prod_{q=1}^{h+1}\delta_{\alpha_{i,q},\lambda_{t,q}}\alpha_{i,h+1}^tf_t & (mod \ N^{m+1}), & s+1\leq i\leq k,  \\[1mm]
[e_i,x_j]\equiv\alpha_{i,j}e_i+\sum\limits_{t=1}^{s_m}\prod_{q=1}^{h}\delta_{\alpha_{i,q},\lambda_{t,q}}\alpha_{i,j}^tf_t & (mod \ N^{m+1}),& s+1\leq i\leq k,\ h+2\leq j \leq s.\\[1mm]
\end{array}\right.
$$

Consider the equalities with $s+1\leq i\leq k,\ h+2\leq j\leq s$
$$0\equiv[e_i,[x_{h+1},x_j]]\equiv[[e_i,x_{h+1}],x_j]-[[e_i,x_j],x_{h+1}]$$
$$\equiv\sum\limits_{t=1}^{s_m}\prod_{q=1}^{h}\delta_{\alpha_{i,q},\lambda_{t,q}}\Big((\alpha_{i,h+1}-\lambda_{t,h+1})\alpha_{i,j}^t+
\delta_{\alpha_{i,h+1},\lambda_{t,h+1}}(\lambda_{t,j}-\alpha_{i,j})\alpha_{i,h+1}^t\Big)f_t \ (mod \ N^{m+1}) ,$$

Since for any $i,t$ ($s+1\leq i\leq k,\ 1\leq t \leq s_m$) there exists $j$ ($1\leq j\leq s$) such that $\alpha_{i,j}\neq \lambda_{t,j}$, from the above congruences we derive
$$\prod_{q=1}^{h}\delta_{\alpha_{i,q},\lambda_{t,q}}(\alpha_{i,h+1}-\lambda_{t,h+1})\alpha_{i,j}^t=0,\quad \prod_{q=1}^{h+1}\delta_{\alpha_{i,q},\lambda_{t,q}}\alpha_{i,h+1}^t=0,\quad 1\leq t\leq s_m.$$
These imply
$$[e_i,x_j]\equiv\alpha_{i,j}e_i \quad (mod \ N^{m+1}), \quad [e_i,x_{h+1}]\equiv\alpha_{i,h+1}e_i \quad (mod \ N^{m+1}), \quad s+1\leq i\leq k, \ 1\leq j\leq h,$$
$$[e_i,x_j]\equiv\alpha_{i,j}e_i+\sum\limits_{t=1}^{s_m}
\prod_{q=1}^{h+1}\delta_{\alpha_{i,q},\lambda_{t,q}}\alpha_{i,j}^tf_t \quad (mod \ N^{m+1}),\quad s+1\leq i\leq k,\ h+2\leq j \leq s.$$

Thus, \eqref{eq144} is done. Now putting $h=s$ in \eqref{eq144} we get the proof of \eqref{eqmain} for $h+1$.  Therefore, \eqref{eqmain} is proved too. Setting in \eqref{eqmain} $m$ to be equal to the nilindex of $N$ we get the products

$$[e_i,x_i]=e_i, \quad 1\leq i\leq s, \quad [e_i,x_j]=\alpha_{i,j}e_i, \quad  s+1\leq i\leq n,\ \ 1\leq j \leq s.$$

%

Applying the arguments for the products $[x_i,x_j]$ as those in Lemma \ref{lem8} we complete the proof of the theorem.
\end{proof}

Now we consider more general case, that is, the case where the nilradical is split.

\begin{prop}\label{lie} Let $\mathcal{R}=\mathcal{N}\oplus \mathcal{Q}$ be a complex maximal extension of finite-dimensional split pure non-characteristically nilpotent Lie algebra with $\mathcal{N}=\bigoplus\limits_{t=1}^{p}\mathcal{N}_t$, where each of $\mathcal{N}_t$ is a non-split and non-characteristically nilpotent ideal of $\mathcal{N}$ which satisfies the condition {\bf A)}. Then $\mathcal{R}$ is a unique (up to isomorphism) Lie algebra and it has the form $\mathcal{R}=\bigoplus\limits_{t=1}^{p}\mathcal{R}_t,$ where $\mathcal{R}_t=\mathcal{N}_t\oplus \mathcal{Q}_t$ is the maximal extension of the nilpotent algebra $\mathcal{N}_t$.
\end{prop}
\begin{proof} To avoid  extra indices, we prove the proposition for the case of $p=2$.
Let $\{e_1, e_2, \dots, e_{s_1},\dots, e_{k_1},e_{k_1+1},\dots, e_{n_1}\}$ and
$\{f_1, f_2, \dots, f_{s_2},\dots, f_{k_2},f_{k_2+1},\dots, f_{n_2}\}$ be the natural bases of $\mathcal{N}_1$ and $\mathcal{N}_2,$ respectively. For arbitrary $d\in Der(\mathcal{N})$ similarly to that of Section \ref{sec3}
we have $d=d_0+d_1$ with diagonalizable and nilpotent derivations of $\mathcal{N}$. Since diagonal elements of $d_0$ (after diagonalization) satisfy the System  $S_{\{e,f\}}$ and $\mathcal{N}$ is split we conclude
$S_{\{e,f\}}=\left\{\begin{array}{ll}
         S_e,\\[1mm]
         S_f.\\[1mm]
        \end{array}\right.$

We set $$rank(S_e)=n_1-s_1, \quad rank(S_f)=n_2-s_2.$$
Then the fundamental system of solutions to $S_{\{e,f\}}$ consists of the following
$$\left(\begin{array}{cccccccccc}
            \alpha_{i_1}&\cdots & 0 & 0 &\cdots&0\\[1mm]
            \vdots&\ddots&\vdots&\vdots&\ddots&\vdots\\[1mm]
            0&\cdots&\alpha_{i_{n_1}}&0&\cdots&0\\[1mm]
            0&\cdots&0&\beta_{j_1}&\cdots&0\\[1mm]
            \vdots&\ddots&\vdots&\vdots&\ddots&\vdots\\[1mm]
            0&\cdots&0&0&\cdots&\beta_{j_{n_2}}\\[1mm]
            \end{array}\right), 1\leq i \leq n_1, 1\leq j \leq n_2.$$
We consider $\mathcal{Q}_1=Span_{\mathbb{C}}\{x_1,\dots,x_{s_1}\}$ and $\mathcal{Q}_2=Span_{\mathbb{C}}\{y_1,\dots,y_{s_2}\}$ such that
$$ad_{x_i}:\left(\begin{array}{cccccccccc}
           \alpha_{i_1}&\cdots & 0 \\[1mm]
           \vdots&\ddots&\vdots\\[1mm]
           0&\cdots&\alpha_{i_{n_1}}\\[1mm]
           \end{array}\right), \ 1\leq i \leq n_1,  \quad
  ad_{y_j}:\left(\begin{array}{cccccccccc}
           \beta_{j_1}&\cdots&0\\[1mm]
           \vdots&\ddots&\vdots\\[1mm]
           0&\cdots&\beta_{j_{n_2}}\\[1mm]
           \end{array}\right), \ 1\leq j \leq n_2, .$$
Applying Theorem \ref{mainthm} for the solvable Lie subalgebras $\mathcal{R}_i=\mathcal{N}_i\oplus \mathcal{Q}_i, \ i=1, 2$ we can assume that
\begin{equation}\left\{\begin{array}{ll}\label{3.16}
[e_i,x_j]=\alpha_{i,j}e_i,& 1\leq i\leq n_1,\  1\leq j\leq s_1, \\[1mm]
[f_i,y_j]=\beta_{i,j}f_i,& 1\leq i\leq n_2,\  1\leq j\leq s_2, \\[1mm]
[x_i,x_j]=\sum\limits_{q=1}^{n_2}\mu_{i,j}^{1,q}f_{q},& 1\leq i,j\leq s_1,\\[1mm]
[y_i,y_j]=\sum\limits_{q=1}^{n_1}\mu_{i,j}^{2,q}e_{q},& 1\leq i,j\leq s_1,\\[1mm]
[x_i,y_j]=\sum\limits_{q=1}^{n_1}\gamma_{i,j}^{q}e_q+\sum\limits_{q=1}^{n_2}\theta_{i,j}^{q}f_q,& 1\leq i\leq s_1,\  1\leq j\leq s_2.\\[1mm]
\end{array}\right.\end{equation}

From the following equalities for $1\leq i,j\leq s_1,\ 1\leq k\leq s_2$:
$$\sum\limits_{q=1}^{n_1}\gamma_{j,k}^{q}\alpha_{q,i}e_q=
[x_i,-\sum\limits_{q=1}^{n_1}\gamma_{j,k}^{q}e_q-\sum\limits_{q=1}^{n_2}
\theta_{j,k}^{q}f_q]=
[x_i,[y_k,x_j]]=[[x_i,y_k],x_j]-[[x_i,x_j],y_k]$$
$$=[\sum\limits_{q=1}^{n_1}\gamma_{i,k}^{q}e_q+\sum\limits_{q=1}^{n_2}\theta_{i,k}^{q}f_q,x_j]-
[\sum\limits_{q=1}^{n_2}\mu_{i,j}^{1,q}f_{q},y_k]=
\sum\limits_{q=1}^{n_1}\alpha_{q,j}\gamma_{i,k}^{q}e_q-\sum\limits_{q=1}^{n_2}\beta_{q,k}
\mu_{i,j}^{1,q}f_{q},$$
we derive
$$[x_i,x_j]=0,\quad  1\leq i,j\leq s_1.$$

Similarly, from the equality
$$[y_i,[x_k,y_j]]=[[y_i,x_k],y_j]-[[y_i,y_j],x_k], \quad 1\leq i,j\leq s_2,\ 1\leq k\leq s_1,$$
we deduce
$$[y_i,y_j]=0,\quad  1\leq i,j\leq s_2.$$

Considering
$$[x_i,y_j]\equiv\sum\limits_{q=1}^{n_2}\theta_{i,j}^{q}f_q, \quad  1\leq i\leq s_1,\  1\leq j\leq s_2 \quad \ (mod \ N_1),$$
we realize that we are in the conditions of Lemma \ref{lem5} (by assuming $u:=x_i, \ 1\leq i\leq s_1$, $\eta:=s_2,\  \tau:=n_2$). Therefore, we obtain
$[x_i,y_j]\equiv 0, \   1\leq i\leq s_1,\  1\leq j\leq s_2 \ (mod \ N_1).$  Hence,
$$[x_i,y_j]=\sum\limits_{q=1}^{n_1}\gamma_{i,j}^{q}e_q, \quad 1\leq i\leq s_1,\  1\leq j\leq s_2.$$

Now applying again Lemma \ref{lem5} to the products
$$[x_i,y_j]=\sum\limits_{q=1}^{n_1}\gamma_{i,j}^{q}e_q, \quad 1\leq i\leq s_1,\  1\leq j\leq s_2.$$
we get
$$[y_j,x_i]=0,\ \  1\leq i\leq s_1,\  1\leq j\leq s_2.$$

Thus,
$$\left\{\begin{array}{ll}\label{3.16}
[e_i,x_j]=\alpha_{i,j}e_i,& 1\leq i\leq n_1,\  1\leq j\leq s_1, \\[1mm]
[f_i,y_j]=\beta_{i,j}f_i,& 1\leq i\leq n_2,\  1\leq j\leq s_2. \\[1mm]
\end{array}\right.$$
\end{proof}

From Theorem \ref{mainthm} and Proposition \ref{lie} we conclude that the complex maximal solvable extension $\mathcal{R}$ of the pure non-characteristically nilpotent nilradical $\mathcal{N}$ which satisfies the condition {\bf A)} admits a basis such that $ad(\mathcal{Q})$ is nothing but the maximal torus of the nilradical. Now applying the result of the conjugacy of maximal torus on $\mathcal{N}$ \cite{Mostow} we obtain partial positive confirmation of \v{S}nobl's conjecture.

Thus, we have proved the following result.

\begin{thm} \label{torus}
A complex maximal solvable extension of a finite-dimensional pure non-characteristically nilpotent Lie algebra which satisfies the condition {\bf A)} is unique (up to isomorphism) and this algebra is isomorphic to the algebra $\mathcal{R}_{\mathcal{T}_{max}}$.
\end{thm}

\section{The vanishing of the first cohomology group for some complex maximal solvable extensions}

In this section we prove that all derivations of maximal  solvable extensions of pure non-characteristically nilpotent Lie algebras which satisfy the condition {\bf A)} are inner. Due to Theorem \ref{torus} we rewrite the table of multiplications of the solvable algebra $\mathcal{R}$ from Theorem \ref{mainthm} as follows:
\begin{equation}\label{eq15}
\left\{\begin{array}{ll}
[e_{\beta_i},e_{\beta_j}]=\sum\limits_{t=1}^{p(\beta_i+\beta_j)}C_{t}(\beta_i,\beta_j)e_{\beta_i+\beta_j}^t,& \beta_i,\beta_j\in W,\\[1mm]
[e_{\alpha_i},x_i]=e_{\alpha_i},& \alpha_i\in \mathbb{C}^*,\ \ 1\leq i\leq s,\\[1mm]
[e_{\beta},x_j]=\beta_{j}e_{\beta},& \beta\in W,\ 1\leq j\leq s,\\[1mm]
\end{array}\right.
\end{equation}
where
\begin{itemize}
  \item $s=dim\mathcal{T}_{max}$;
  \item $W$ is the set of roots of $N$ with respect to the action of $\mathcal{T}_{max}$;
  \item $e_{\beta_i+\beta_j}^t$ are a basis elements of the root space $N_{\beta_i+\beta_j}$;
  \item $dimN_{\beta_i+\beta_j}=p(\beta_i+\beta_j)$;
  \item $\alpha_i$ are primitive roots (they generate the rest roots).
\end{itemize}

\begin{thm} \label{thmH1eq0} Let $\mathcal{R}$ be a complex maximal solvable extension of a pure non-characteristically nilpotent Lie algebra $\mathcal{N}$ which satisfies the condition {\bf A)}. Then any derivation of $\mathcal{R}$ is inner.
\end{thm}
\begin{proof} Thanks to Theorem \ref{mainthm} we conclude that the algebra $\mathcal{R}$ satisfies the conditions of Theorem \ref{thmSerre}. Therefore, applying Equality (\ref{eq222}) for the case of $p=1$ we have
\begin{equation}
H^1(\mathcal{R},\mathcal{R})\simeq H^1(\mathcal{Q},\mathbb{C})\otimes H^0(\mathcal{N},\mathcal{R})^{\mathcal{Q}}+H^0(\mathcal{Q},\mathbb{C})\otimes H^1(\mathcal{N},\mathcal{R})^{\mathcal{Q}}.
\end{equation}

Due to the table of multiplications in Theorem \ref{mainthm} we easily get $Center(\mathcal{R})=0$ and $H^0(\mathcal{N},\mathcal{R})^{\mathcal{Q}}=0.$

For an arbitrary $d\in Z^1(\mathcal{N},\mathcal{R})^{\mathcal{Q}}$ we set
$$d(e_{\alpha_i})=\sum\limits_{t=1}^{s}A_{\alpha_i,t}e_{\alpha_i}+\sum\limits_{\beta\in W}^{}A_{\alpha_i,\beta}e_{\beta}+
\sum\limits_{t=1}^{s}B_{\alpha_i,t}x_t,\quad 1\leq i\leq s.$$

Assuming $d(\mathcal{Q})=0$ we have
$$[d(a),b]+[a,d(b)]-d([a,b])=0,\quad a,b\in \mathcal{R}.$$

Taking into account the table of multiplications of $\mathcal{R}$ and \eqref{eq14}, for $1\leq i\leq s$ we have
$$\begin{array}{ll}
0&=[x_i,d(e_{\alpha_i})]-d([x_i,e_{\alpha_i}]) \\[1mm]
&=[x_i,\sum\limits_{t=1}^{s}A_{\alpha_i,t}e_{\alpha_i}+\sum\limits_{\beta\in W}^{}A_{\alpha_i,\beta}e_{\beta}+\sum\limits_{t=1}^{s}B_{\alpha_i,t}x_t]+d(e_{\alpha_i})\\[1mm]
&=-A_{\alpha_i,i}e_{\alpha_i}-\sum\limits_{\beta\in W}^{}\beta_{i}A_{\alpha_i,\beta}e_{\beta}+
\sum\limits_{t=1}^{s}A_{\alpha_i,t}e_{\alpha_i}+\sum\limits_{\beta\in W}^{}A_{\alpha_i,\beta}e_{\beta}+\sum\limits_{t=1}^{s}B_{\alpha_i,t}x_t\\[1mm]
&=\sum\limits_{t=1,\ t\neq i}^{s}A_{\alpha_i,t}e_{\alpha_i}+\sum\limits_{\beta\in W}^{}(1-\beta_{i})A_{\alpha_i,\beta}e_{\beta}+
\sum\limits_{t=1}^{s}B_{\alpha_i,t}x_t.
\end{array}$$

Consequently, $A_{\alpha_i,t}=B_{\alpha_i,t}=B_{\alpha_i,i}=0,\ 1\leq i\neq t\leq s$ and
\begin{equation}\label{eqab}
(1-\beta_{i})A_{\alpha_i,\beta}=0,\quad \beta\in W.
\end{equation}
Therefore,
$$d(e_{\alpha_i})=A_{\alpha_i,i}e_{\alpha_i}+\sum\limits_{\beta\in W}^{}A_{\alpha_i,\beta}e_{\beta}.$$

Using \eqref{eq14}, for $1\leq i\neq j\leq s$ we get the chain of equalities
$$0=[x_i,d(e_{\alpha_j})]-d([x_i,e_{\alpha_j}])=[x_i,A_{\alpha_j,j}e_{\alpha_j}+\sum\limits_{\beta\in W}^{}A_{\alpha_j,\beta}e_{\beta}]=
-\sum\limits_{\beta\in W}^{}A_{\alpha_j,\beta}\beta_ie_{\beta}.$$
Taking into account the existence $i$ such that $\beta_i\neq 0$ and \eqref{eqab} we get $A_{\alpha_j,\beta}=0$ for $1\leq j\leq s, \ \beta\in W.$

Therefore, $d(e_{\alpha_i})=A_{\alpha_i,i}e_{\alpha_i}, \ 1\leq i\leq s$ and hence, $dimZ^1(\mathcal{N},\mathcal{R})^{\mathcal{Q}}\leq s.$

Since $\mathcal{N}$ is ideal, for all $z\in \mathcal{R}$ we get $ad_z \in B^1(\mathcal{N}, \mathcal{R})$. Moreover, for each $\beta\in W$ there exists $j$ ($1 \leq j \leq s$) such that $ad_{e_{\beta}}(x_j)\neq 0$. Hence, $ad_{e_{\beta}}\notin dimB^1(\mathcal{N}, \mathcal{R})^{\mathcal{Q}}.$

From the table of multiplications of $\mathcal{R}$ we derive $ad_{x_i}(x_j)=0$ for any $x_i, x_j\in \mathcal{Q}$ and $ad_{e_i}(\mathcal{Q})\neq0$ for any $i$ ($1\leq i \leq n$), which imply $dimB^1(\mathcal{N},\mathcal{R})^{\mathcal{Q}}=s.$

The inequalities $s=dimB^1(\mathcal{N},\mathcal{R})^{\mathcal{Q}}\leq dimZ^1(\mathcal{N},\mathcal{R})^{\mathcal{Q}} \leq s$ give $H^1(\mathcal{N},\mathcal{R})^{\mathcal{Q}}=0,$
which along with $H^0(\mathcal{N},\mathcal{R})^{\mathcal{Q}}=0$ completes the proof of the theorem.
\end{proof}

Further we make use the following result from \cite{Leger1}.

\begin{prop}\label{prop1}  Let $\mathcal{L}$ be a Lie algebra over a field of characteristic $0$ such that $Der(\mathcal{L})=Inder(\mathcal{L}).$ If the center of the Lie algebra $\mathcal{L}$ is non trivial, then $\mathcal{L}$ is not solvable and the radical of $\mathcal{L}$ is nilpotent.
\end{prop}
\begin{prop} \label{thmH1neq0} Any non maximal solvable extension of a nilpotent Lie algebra $\mathcal{N}$ (not characteristically nilpotent) admits an outer derivation.
\end{prop}

\begin{proof} Let $\mathcal{R}=\mathcal{N}\oplus \mathcal{Q}$ with $\mathcal{Q}=Span_{\mathbb{C}}\{x_1, \dots, x_s\}$ and $s < dim\mathcal{T}_{max}$.

Consider the following two options.

\emph{\textbf{Case 1.}} Let  $Center(\mathcal{R})=\{0\}$. Since $ad(\mathcal{R})$ is a solvable Lie algebra (as homomorphic image of solvable Lie algebra $\mathcal{R}$), by Lie's Theorem $ad_r$ ($r \in \mathcal{R}$) has upper-triangular form and for any $x_i\in \mathcal{Q}$ we have the decomposition
$$ad_{x_i}=d_i+d_{n_i},\ \quad 1\leq i\leq s,$$
where $d_i: \mathcal{R} \to \mathcal{R}$ is a diagonal derivation and $d_{n_i} : \mathcal{R} \to \mathcal{R}$ is nilpotent derivation whose matrix is strictly upper-triangular.

If there exists $i_0$ ($1\leq i_0 \leq s$) such that $d_{n_{i_0}}\notin Inder(\mathcal{R})$. Then $d_{n_{i_0}}$ is outer derivation of $\mathcal{R}$. If $d_{n_{i}}\in Inder(\mathcal{R})$ for any $1\leq i\leq s,$ that is, there are $y_i \in \mathcal{R}$ such that $ad_{y_i}=d_{n_{i}}$.
Then, ${d_i}_{|\mathcal{N}}={ad_{x_i}}_{|\mathcal{N}}-{ad_{y_i}}_{|\mathcal{N}}={ad_{x_i-y_i}}_{|\mathcal{N}}$ lies in some maximal torus of $\mathcal{N}$. Since $s < dim\mathcal{T}_{max}$,
there exists $\tilde{d}\in \mathcal{T}_{max}\setminus Span\{d_1,\ldots, d_s\}.$

We set $\mathcal{Q}'=Span_{\mathbb{C}}\{x_i'=x_i-y_i, \ 1\leq i \leq s\}.$ Then $\mathcal{R}=\mathcal{N}\oplus \mathcal{Q}'$ with $dim \mathcal{Q}'<s$.

The equalities
$$ad_{[x_i',x_j']}(z)=[ad_{x_i'},ad_{x_j'}](z)=[d_i,d_j](z)=0,$$
for any $z\in \mathcal{R}$ imply $[x_i',x_j']\in Center(\mathcal{R})=\{0\}$, hence $[\mathcal{Q}', \mathcal{Q}']=0$. Consequently, for any $x\in \mathcal{Q}'$ the matrix of the operator $ad_x$ has diagonal form (as a linear combination of diagonal matrices).

We set $D(x_i')=0,\ 1\leq i\leq s$ and $D_{|\mathcal{N}}\equiv \tilde{d}$.

From the chain of equalities
$$D([n,x])-[n,D(x)]-[D(n),x]=D([n,x])-[D(n),x]=-D(ad_{x}(n))+ad_{x}(D(n))=$$
$$[ad_{x},D](n)=0 \quad n\in \mathcal{N}, x\in \mathcal{Q}',$$
and construction of $D$ we conclude that $D$ is outer derivation.

\emph{\textbf{Case 2.}}  Let $Center(\mathcal{R})\neq \{0\}$. Suppose that $Der(\mathcal{R})=Inder(\mathcal{R})$, then by Proposition \ref{prop1} we conclude that algebra $\mathcal{R}$ is nilpotent, that is a contradiction. Therefore, $Inder(\mathcal{R})\subsetneqq Der(\mathcal{R})$.
\end{proof}

%

\section{Comparisons with some known descriptions of solvable Lie algebras.}\label{Sec5}

In this section we give several applications of Theorem \ref{mainthm} to already known classification results.

\begin{exam}\label{com51}
Let $\mathcal{N}=Q_{2n}$ with the table of multiplications:
$$[e_1,e_k]=e_{k+1}, \ 2 \leq k \leq 2n-2, \quad [e_k,e_{2n+1-k}]=(-1)^ke_{2n},\ 2 \leq k \leq n.$$
Here is a result from \cite{ancochea}.

\begin{prop}  For any $n\geq 3$ there is a unique (up to isomorphism) solvable Lie algebra of dimension $(2n+2)$ having $Q_{2n}$ as a nilradical:
$$r_{2n+2}:\quad \left\{\begin{array}{lll}
[e_1,e_k]=e_{k+1}, \ 2 \leq k \leq 2n-2,& [e_k,e_{2n+1-k}]=(-1)^ke_{2n},\ 2 \leq k \leq n ,\\[1mm]
[x_1,e_{k}]=ke_{k},\ 1 \leq k \leq 2n-1, & [x_1,e_{2n}]=(2n+1)e_{2n},\\[1mm]
[x_2,e_{k}]=e_{k},\ 1 \leq k \leq 2n-1, & [x_2,e_{2n}]=2e_{2n}.\\[1mm]
\end{array}\right.$$
\end{prop}
Applying Theorem \ref{mainthm} we get a unique solvable Lie algebra with the nilradical $Q_{2n}$:
$$R(Q_{2n}): \quad \left\{\begin{array}{lll}
[e_1,e_k]=e_{k+1}, \ 2 \leq k \leq 2n-2,& [e_k,e_{2n+1-k}]=(-1)^ke_{2n},\ 2 \leq k \leq n ,\\[1mm]
[e_{1},x_1]=x_{1},& [e_k,x_1]=(k-2)e_k, \ 3 \leq k \leq 2n-1, \\[1mm]
[e_{2n},x_1]=(2n-3)e_{2n},& [e_k,x_2]=e_k, \ 2 \leq k \leq 2n-1, \\[1mm]
[e_{2n},x_2]=2e_{2n}.\\[1mm]
\end{array}\right.$$

If in this algebra we take the change
$$x_1'=-x_1-2y_2,\quad x_2'=-x_2,$$
then we get the algebra $r_{2n+2}$.
\end{exam}

\begin{exam}\label{com52} Let us consider the nilpotent Lie algebra
$$\bar{N}: [e_n,e_i]=e_{i-1}, \quad 2 \leq i \leq n-2.$$
Clearly, $e_n,e_{n-1},e_{n-2}$ are generator basis elements of $\bar{N}$.

In the paper \cite{WaLiDe} the following classification theorem was proved:
\begin{thm}\label{thm416} Only one equivalence class of solvable Lie algebras of $dim g= n +
3$ with the nilradical $\bar{N}$ exists. It can be represented by Lie brackets as follows:
$$\left\{\begin{array}{lll}
[e_n,e_i]=e_{i-1}, &2 \leq i \leq n-2, &\\[1mm]
[h_1,e_{n}]=e_{n}, & [h_1,e_i]=(n-i-2)e_{i}, & 1 \leq i \leq n-3,\\[1mm]
[h_3,e_{n-1}]=e_{n-1},& [h_2,e_i]=e_{i}, & 1 \leq i \leq n-2, \\[1mm]
[h_1,h_2]=ae_{n-1},& [h_1,h_3]=be_{1},& [h_2,h_3]=0.\\[1mm]
\end{array}\right.$$
\end{thm}

In fact, in this algebra we have $a=b=0$ (because of Jacobi identity for the triples of elements
$\{h_1,h_2,h_3\}$).

From Theorem \ref{mainthm} we obtain
$$R(\bar{N}):\left\{\begin{array}{lll}
$$\begin{array}{lll}
[e_n,e_i]=e_{i-1}, &2 \leq i \leq n-2, &\\[1mm]
[e_i,x_1]=(n-i-2)e_{i}, & 1 \leq i \leq n-3,\\[1mm]
[e_i,x_2]=e_{i}, & 1 \leq i \leq n-2,  \\[1mm]
[e_{n},x_1]=e_{n},& \\[1mm]
[e_{n-1},x_3]=e_{n-1}.& \\[1mm]
\end{array}$$
\end{array}\right.$$

If we take base change
$$h_1=-x_1,\quad h_2=-x_2,\quad h_3=-x_3,$$
then we get the algebra of Theorem  \ref{thm416}.
\end{exam}

\begin{exam}\label{com53} Consider the nilpotent Lie algebra
$$n_{n,1}:\quad [e_k,e_n]=e_{k-1},\quad 2\leq k \leq n-1.$$

In the paper \cite{snoble} (Theorem 13.2) the following classification theorem was proved:

\begin{thm} \label{thm3.21} Any solvable, indecomposable, non nilpotent Lie algebra $s$ with
the nilradical $n_{n,1}$ has dimension $n + 2$.
Precisely, only one class of solvable Lie algebras $s_{n+2}$ of dim s =
n + 2 with the nilradical $n_{n,1}$ exists. It is represented by a
basis $\{e_1,\dots , e_n, f_1, f_2\}$ and the Lie brackets
involving $f_1$ and $f_2$ are

$$\left\{\begin{array}{lll}
[f_1, e_n] = e_n,& [f_2, e_n] = [f_1, f_2] = 0,&\\[1mm]
[f_1,e_k] =(n-1-k)e_k, & 1\leq  k \leq n-1,& \\[1mm]
[f_2,e_k] = e_k, &  1\leq k \leq n-1.& \\[1mm]
\end{array}\right.$$
\end{thm}

Applying Theorem \ref{mainthm} we derive
$$R(n_{n,1}):\quad \left\{\begin{array}{lll}
[e_k,e_n]=e_{k-1},& 2\leq k \leq n-1, \\[1mm]
[e_i,x_1]=(n-i-1)e_i, & 1\leq i\leq n-2,\\[1mm]
[e_i,x_2]=e_i, & 1\leq i\leq n-1, \\[1mm]
[e_n,x_1]=e_n.&\\[1mm]
\end{array}\right.$$

This algebra is isomorphic to the algebra of Theorem \ref{thm3.21}, because of the change
$$f_1=-x_1,\quad f_2=-x_2.$$

\end{exam}

\begin{exam}\label{com54} In the paper \cite{snoble} the following solvable Lie algebra with two-dimensional abelian nilradical was obtained:

$$s_{4,12}:\quad \left\{\begin{array}{lll}
[e_1,e_3]=e_1, & [e_2,e_3]=e_2,\\[1mm]
[e_1,e_4]=-e_2, &[e_2,e_4]=e_1.\\[1mm]
\end{array}\right.$$

Using Proposition \ref{lie} and Theorem \ref{mainthm} we deduce the algebra
$$[e_1,x_1]=e_1,\quad [e_2,x_2]=e_2.$$

This algebra is isomorphic to the algebra $s_{4,12}$. Indeed, we only need to consider the following base change:
$$e_1'=ie_1-ie_2,\ e_2'=e_1+e_2,\ e_3'=x_1+x_2,\ e_4'=ix_1-ix_2.$$
\end{exam}

\begin{exam}\label{com55} Let us consider four-dimensional nilpotent Lie algebra
$$g:\quad [e_2,e_4]=e_1, \quad [e_3,e_4]=e_2.$$

In the monograph \cite{snoble} the following solvable Lie algebra with the nilradical $g$ was presented:
$$s_{6,242}: \left\{\begin{array}{llll}
[e_2,e_4]=e_1,& [e_3,e_4]=e_2,& [e_5,e_1]=2e_1,& [e_5,e_2]=e_2,\\[1mm]
[e_5,e_4]=e_4,& [e_6,e_1]=e_1,& [e_6,e_2]=e_2,& [e_6,e_3]=e_3.\\[1mm]
\end{array}\right.$$

Again applying Theorem \ref{mainthm} we have the solvable Lie algebra with the nilradical $g$:
$$\left\{\begin{array}{llll}
[e_2,e_4]=e_1,& [e_3,e_4]=e_2,& [e_1,x_1]=2e_1,& [e_2,x_1]=e_2,\\[1mm]
[e_4,x_1]=e_4,& [e_1,x_2]=e_1,& [e_2,x_2]=e_2,& [e_3,x_2]=e_3.\\[1mm]
\end{array}\right.$$

This algebra is nothing but the algebra $s_{6,242}$ (assuming $e_5=-x_1$ and $e_6=-x_2$).
\end{exam}

\section{The vanishing of the cohomology groups for some solvable Lie algebras}

In this section we estimate the dimension of the space $H^m(\mathcal{N},\mathcal{R})^{\mathcal{Q}}$ ($m\geq 0$), where $\mathcal{R}$ is a maximal solvable extension of a pure non-characteristically nilpotent Lie algebra $\mathcal{N}$ and prove that $H(\mathcal{R},\mathcal{R})=0$ under some conditions on the nilradical $\mathcal{N}$.

Let $\varphi\in Z^{m}(\mathcal{N},\mathcal{R})^{\mathcal{Q}}$, then $\varphi\in C^{m}(\mathcal{N},\mathcal{R})$ such that

\begin{equation} \label{eqZ^m}\begin{array}{lll}
0 &=(d^m\varphi)(e_{i_1}, \dots , e_{i_{m+1}})\\[1mm]
 & = \sum\limits_{s=1}^{m+1}(-1)^{s+1}[e_s,\varphi(e_{i_1},\dots, \widehat{e}_{i_s}, \dots , e_{i_{m+1}})]\\[1mm]
 & +\sum\limits_{1\leq s<t\leq {m+1}}(-1)^{s+t}\varphi([e_{i_s},e_{i_t}], e_{i_1}, \dots,\widehat{e}_{i_s}, \dots, \widehat{e}_{i_t}, \dots, e_{i_{m+1}})
\end{array}\end{equation}
and
\begin{equation}\label{Q-inv} (z.\varphi)(e_{i_1},\ldots,e_{i_m})=0,\ \ z\in \mathcal{Q},
\end{equation}
where
$(z.\varphi)(e_{i_1},\ldots,e_{i_m})=[z,\varphi(e_{i_1},\ldots,e_{i_m})]-
\sum\limits_{s=1}^{m}\varphi(e_{i_1},\ldots,[z,e_{i_s}],\ldots,e_{i_m}).$

Let us introduce the notation

$$\varphi(e_{i_1},\ldots,e_{i_m})=\sum_{s=1}^{n}\alpha^{\varphi,s}_{\{{i_1}, \dots, {i_m}\}} e_s+ x_{\{{i_1}, \dots, {i_m}\}}, \quad x_{\{{i_1}, \dots, {i_m}\}}\in \mathcal{Q}.$$

Consider (\ref{Q-inv}) in the following form
$$0=(z.\varphi)(e_{\delta_1},\ldots,e_{\delta_m})=[z,\varphi(e_{\delta_1},\ldots,e_{\delta_m})]-
\sum\limits_{s=1}^{m}\varphi(e_{\delta_1},\ldots,[z,e_{\delta_s}],\ldots,e_{\delta_m})=
[z,\varphi(e_{\delta_1},\ldots,e_{\delta_m})]$$
$$+\sum\limits_{s=1}^{m}\delta_s\varphi(e_{\delta_1},\ldots,e_{\delta_s},\ldots,e_{\delta_m})=
-[\varphi(e_{\delta_1},\ldots,e_{\delta_m}),z]+
(\sum\limits_{s=1}^{m}\delta_s)\varphi(e_{\delta_1},\ldots,e_{\delta_s},\ldots,e_{\delta_m}).
$$

This implies  that
$$[\varphi(e_{\delta_1},\ldots,e_{\delta_m}),z]=
(\sum\limits_{s=1}^{m}\delta_s)\varphi(e_{\delta_1},\ldots,e_{\delta_s},\ldots,e_{\delta_m}),
$$
that is, $\varphi(\mathcal{N}_{\delta_1},\ldots,\mathcal{N}_{\delta_m})\subseteq \mathcal{N}_{\sum\limits_{s=1}^{m}\delta_s}.$

We set

$$\begin{array}{rl}
S\delta_{{m+1}}:&={\sum_{p=1}^{m+1}\delta_p},\\[1mm]
S\delta_{{m+1}\setminus s}:&={\sum_{p=1, p\neq s}^{m+1}\delta_p},\\[1mm]
[e_{\delta_1},\dots, \widehat{e}_{\delta_s}, \dots , e_{\delta_{m+1}}]&=\sum_{p=1}^{dimN_{S\delta_{{m+1}\setminus s}}}
A_{{\delta_1, \dots,\widehat{{\delta_s}}, \dots, \delta_{m+1}}}^p e^p_{S\delta_{{m+1}\setminus s}},\\[1mm]
[e^p_{S\delta_{{m+1}}},e_{\delta_s}]&=\sum_{q=1}^{dimN_{S\delta_{m+1}}}A_{S\delta_{{m+1}\setminus s}, \delta_s}^{p,q} e^q_{S\delta_{{m+1}}},\\[1mm]
\varphi(e_{\delta_1},\dots, \widehat{e}_{\delta_s}, \dots , e_{\delta_{m+1}})&=\sum_{p=1}^{dimN_{S\delta_{{m+1}\setminus s}}}B_{{\delta_1, \dots,\widehat{{\delta_s}}, \dots,  \delta_{m+1}}}^p e^p_{S\delta_{{m+1}\setminus s}},\\[1mm]
\varphi(e^p_{\delta_s+\delta_t}, e_{\delta_1}, \dots,
\widehat{e}_{\delta_s}, \dots, \widehat{e}_{\delta_t}, \dots, e_{\delta_{m+1}})&=\sum_{q=1}^{dimN_{S\delta_{{m+1}}}}B_{{\delta_s+\delta_t, \delta_1, \dots,\widehat{{\delta_s}}, \dots, \widehat{{\delta_t}}, \dots, \delta_{m+1}}}^{p,q} e^q_{S\delta_{{m+1}}}.
\end{array}$$

Now consider (\ref{eqZ^m}) in terms of the notations above

$$\begin{array}{ll}
0&=\sum\limits_{s=1}^{m+1}(-1)^{s+1}[e_{\delta_s},\varphi(e_{\delta_1},
\dots, \widehat{e}_{\delta_s}, \dots , e_{\delta_{m+1}})]\\[3mm]
&+\sum\limits_{1\leq s<t\leq {m+1}}(-1)^{s+t}\varphi([e_{\delta_s},e_{\delta_t}], e_{\delta_1}, \dots,
\widehat{e}_{\delta_s}, \dots, \widehat{e}_{\delta_t}, \dots, e_{\delta_{m+1}})\\[3mm]
&=\sum\limits_{s=1}^{m+1}(-1)^{s+1}\sum_{p=1}^{dimN_{S\delta_{{m+1}\setminus s}}}B_{{\delta_1, \dots,\widehat{{\delta_s}}, \dots,  \delta_{m+1}}}^p
[e_{\delta_s},e^p_{S\delta_{{m+1}\setminus s}}]\\[3mm]
&+\sum\limits_{1\leq s<t\leq {m+1}}(-1)^{s+t}\sum_{p=1}^{dimN_{\delta_s+\delta_t}}A_{\delta_s, \delta_t}^p
\varphi(e^p_{\delta_s+\delta_t}, e_{\delta_1}, \dots,\widehat{e}_{\delta_s}, \dots, \widehat{e}_{\delta_t}, \dots, e_{\delta_{m+1}})\\[3mm]
&=\sum\limits_{s=1}^{m+1}(-1)^{s}\sum_{p=1}^{dimN_{S\delta_{{m+1}\setminus s}}}B_{{\delta_1, \dots,\widehat{{\delta_s}}, \dots,  \delta_{m+1}}}^p\sum_{q=1}^{dimN_{S\delta_{m+1}}}A_{S\delta_{{m+1}\setminus s}, \delta_s}^{p,q} e^q_{S\delta_{{m+1}}}\\[3mm]
&+\sum\limits_{1\leq s<t\leq {m+1}}(-1)^{s+t}\sum_{p=1}^{dimN_{\delta_s+\delta_t}}A_{\delta_s, \delta_t}^p
\sum_{q=1}^{dimN_{S\delta_{{m+1}}}}B_{{\delta_s+\delta_t, \delta_1, \dots,\widehat{{\delta_s}}, \dots, \widehat{{\delta_t}}, \dots, \delta_{m+1}}}^{p,q} e^q_{S\delta_{{m+1}}}\\[3mm]
&=\sum_{q=1}^{dimN_{S\delta_{m+1}}}[\sum\limits_{s=1}^{m+1}(-1)^{s}\sum_{p=1}^{dimN_{S\delta_{{m+1}\setminus s}}}B_{{\delta_1, \dots,\widehat{{\delta_s}}, \dots,  \delta_{m+1}}}^pA_{S\delta_{{m+1}\setminus s}, \delta_s}^{p,q}\\[3mm]
&+\sum\limits_{1\leq s<t\leq {m+1}}(-1)^{s+t}\sum_{p=1}^{dimN_{\delta_s+\delta_t}}A_{\delta_s, \delta_t}^p
B_{{\delta_s+\delta_t, \delta_1, \dots,\widehat{{\delta_s}}, \dots, \widehat{{\delta_t}}, \dots, \delta_{m+1}}}^{p,q}] e^q_{S\delta_{{m+1}}}.
\end{array}$$

From this for any $q$ such that $1\leq q\leq dim\mathcal{N}_{S\delta_{m+1}}$ we obtain the system of equalities

\begin{equation}\label{hm1}
S_m(\mathcal{R}):\left\{\begin{array}{ll}
\sum\limits_{s=1}^{m+1}(-1)^{s}\sum\limits_{p=1}^{dim\mathcal{N}_{S\delta_{{m+1}\setminus s}}}B_{{\delta_1, \dots,\widehat{{\delta_s}},
\dots,  \delta_{m+1}}}^pA_{S\delta_{{m+1}\setminus s}, \delta_s}^{p,q} \\[1mm]
+\sum\limits_{1\leq s<t\leq {m+1}}(-1)^{s+t}\sum\limits_{p=1}^{dim\mathcal{N}_{\delta_s+\delta_t}}A_{\delta_s, \delta_t}^p
B_{{\delta_s+\delta_t, \delta_1, \dots,\widehat{{\delta_s}}, \dots, \widehat{{\delta_t}}, \dots, \delta_{m+1}}}^{p,q}=0.\\[1mm]
\end{array}\right.\end{equation}

Let us check that the equality \eqref{hm1} is not changed if we replace $\delta_i$ to $\delta_{i+l}$ (similarly, if replace $\delta_{i+l}$ to
$\delta_{i}$). Obviously, it is enough to check for $\delta_i$ to $\delta_{i+1}.$

Consider

$$\begin{array}{ll}
&\sum\limits_{s=1}^{m+1}(-1)^{s}\sum_{p=1}^{dim\mathcal{N}_{S\delta_{{m+1}\setminus s}}}B_{{\delta_1, \dots,\delta_{i+1}, \delta_i, \dots, \widehat{{\delta_s}}, \dots,  \delta_{m+1}}}^pA_{S\delta_{{m+1}\setminus s}, \delta_s}^{p,q}\\[3mm]
+&\sum\limits_{1\leq s<t\leq {m+1}}(-1)^{s+t}\sum_{p=1}^{dim\mathcal{N}_{\delta_s+\delta_t}}A_{\delta_s, \delta_t}^pB_{{\delta_s+\delta_t, \delta_1, \dots,\delta_{i+1}, \delta_i,\dots, \widehat{{\delta_s}}, \dots, \widehat{{\delta_t}}, \dots, \delta_{m+1}}}^{p,q}\\[3mm]
=&\sum\limits_{s=1, s\notin\{i, i+1\}}^{m+1}(-1)^{s}\sum_{p=1}^{dim\mathcal{N}_{S\delta_{{m+1}\setminus s}}}B_{{\delta_1, \dots,\delta_{i+1}, \delta_i,  \dots, \widehat{{\delta_s}}, \dots,  \delta_{m+1}}}^pA_{S\delta_{{m+1}\setminus s}, \delta_s}^{p,q}\\[3mm]
+&(-1)^{i}\sum_{p=1}^{dim\mathcal{N}_{S\delta_{{m+1}\setminus i}}}B_{{\delta_1, \dots,\delta_{i-1}, \delta_i, \widehat{{\delta_{i+1}}}, \dots,  \delta_{m+1}}}^pA_{S\delta_{{m+1}\setminus i+1}, \delta_i}^{p,q}\\[3mm]
+&(-1)^{i+1}\sum_{p=1}^{dim\mathcal{N}_{S\delta_{{m+1}\setminus i+1}}}B_{{\delta_1, \widehat{{\delta_{i}}},\dots,\delta_{i+1}, \delta_{i+2}, \dots,  \delta_{m+1}}}^pA_{S\delta_{{m+1}\setminus i}, \delta_{i+1}}^{p,q}\\[3mm]
+&\sum\limits_{1\leq s<t\leq {m+1}, (s,t)\neq (i,i+1)}(-1)^{s+t}\sum_{p=1}^{dim\mathcal{N}_{\delta_s+\delta_t}}A_{\delta_s, \delta_t}^p
B_{{\delta_s+\delta_t, \delta_1, \dots,\delta_{i+1}, \delta_i,\dots, \widehat{{\delta_s}}, \dots, \widehat{{\delta_t}}, \dots, \delta_{m+1}}}^{p,q}\\[3mm]
+&(-1)^{2i+1}\sum_{p=1}^{dim\mathcal{N}_{\delta_s+\delta_t}}A_{\delta_{i+1}, \delta_i}^pB_{{\delta_{i+1}+\delta_i, \delta_1, \dots,\widehat{{\delta_{i+1}}}, \widehat{{\delta_i}}, \dots, \delta_{m+1}}}^{p,q}\\[3mm]
=&-\sum\limits_{s=1, s\notin\{i, i+1\}}^{m+1}(-1)^{s}\sum_{p=1}^{dim\mathcal{N}_{S\delta_{{m+1}\setminus s}}}B_{{\delta_1, \dots,\delta_{i}, \delta_{i+1},  \dots, \widehat{{\delta_s}}, \dots,  \delta_{m+1}}}^pA_{S\delta_{{m+1}\setminus s}, \delta_s}^{p,q}\\[3mm]
-&(-1)^{i+1}\sum_{p=1}^{dim\mathcal{N}_{S\delta_{{m+1}\setminus i}}}B_{{\delta_1, \dots,\delta_{i-1}, \delta_i, \widehat{{\delta_{i+1}}}, \dots,  \delta_{m+1}}}^pA_{S\delta_{{m+1}\setminus i+1}, \delta_i}^{p,q}\\[3mm]
-&(-1)^{i}\sum_{p=1}^{dim\mathcal{N}_{S\delta_{{m+1}\setminus i+1}}}B_{{\delta_1, \dots,\widehat{{\delta_{i}}}, \delta_{i+1}, \delta_{i+2}, \dots,  \delta_{m+1}}}^pA_{S\delta_{{m+1}\setminus i}, \delta_{i+1}}^{p,q}\\[3mm]
-&\sum\limits_{1\leq s<t\leq {m+1}, (s,t)\neq (i,i+1)}(-1)^{s+t}\sum_{p=1}^{dim\mathcal{N}_{\delta_s+\delta_t}}A_{\delta_s, \delta_t}^p
B_{{\delta_s+\delta_t, \delta_1, \dots,\delta_{i}, \delta_{i+1},\dots, \widehat{{\delta_s}}, \dots, \widehat{{\delta_t}}, \dots, \delta_{m+1}}}^{p,q}\\[3mm]
-&(-1)^{2i+1}\sum_{p=1}^{dim\mathcal{N}_{\delta_{i+1}+\delta_i}}A_{\delta_i,\delta_{i+1}}^p
B_{{\delta_{i+1}+\delta_i, \delta_1, \dots, \widehat{{\delta_i}}, \widehat{{\delta_{i+1}}}, \dots, \delta_{m+1}}}^{p,q}\\[1mm]
=&-\sum\limits_{s=1}^{m+1}(-1)^{s}\sum_{p=1}^{dim\mathcal{N}_{S\delta_{{m+1}\setminus s}}}B_{{\delta_1, \dots,\widehat{{\delta_s}}, \dots,  \delta_{m+1}}}^pA_{S\delta_{{m+1}\setminus s}, \delta_s}^{p,q}\\[1mm]
-&\sum\limits_{1\leq s<t\leq {m+1}}(-1)^{s+t}\sum_{p=1}^{dim\mathcal{N}_{\delta_s+\delta_t}}A_{\delta_s, \delta_t}^p
B_{{\delta_s+\delta_t, \delta_1, \dots,\widehat{{\delta_s}}, \dots, \widehat{{\delta_t}}, \dots, \delta_{m+1}}}^{p,q}=0.
\end{array}$$

Therefore, the number of equalities in \eqref{hm1} is reduced to the number of
the pair $(s,t)$ such that $1\leq s <t\leq m+1$ multiplied to $dim \mathcal{N}_{\delta_1+\dots, \delta_{m+1}}$, where $m+1$-tuple with paiwise different roots in $dim N_{\delta_1+\dots, \delta_{m+1}}$ is considered only one time.

Therefore, the number of independent equalities in System \eqref{hm1} is bounded the above by
\begin{equation}\label{criteriya8}
\mathcal{S}_m:=\sum\limits_{\delta_i\neq\delta_j}dim\mathcal{N}_{\delta_1}dim\mathcal{N}_{\delta_2}\cdots dim\mathcal{N}_{\delta_{m+1}}dim \mathcal{N}_{\delta_1+\dots+ \delta_{m+1}}.
\end{equation}

In combinatorics \cite{combinatorics} the general form of the principle of inclusion-exclusion states that for finite sets $A_1,\ldots, A_n$, one has the following formula
\begin{equation}\label{inclusion}
  \left|\bigcup\limits_{i=1}^{n}A_{i}\right|=\sum\limits_{k=1}^{n}(-1)^{k+1}\left(\sum\limits_{1\leq i_{1}<\cdots <i_{k}\leq n}|A_{i_{1}}\cap \cdots \cap A_{_{k}}|\right).
\end{equation}

Applying \eqref{inclusion} to the roots subspaces $\mathcal{N}_{\alpha}$ we find the number of free parameters defining an arbitrary element of the space $C^m(\mathcal{N},\mathcal{R})^\mathcal{Q}$  $\varphi$. Namely, we have
\begin{equation}\begin{array}{ll}\label{criteriya9}
dimC^m(\mathcal{N},\mathcal{R})^{\mathcal{Q}}&=\sum\limits_{\delta_1\in W}dim\mathcal{N}_{\delta_1}(dim \mathcal{N}_{\delta_1}-1)\cdots (dim\mathcal{N}_{\delta_{1}}-m+1)dim\mathcal{N}_{m\delta_1}\\[1mm]
&+\sum\limits_{i=1}^{m-1}\sum\limits_{\begin{array}{ll}
1\leq r_1\leq\dots\leq r_i,\\[1mm]
\sum r_i<m,\ \delta_i\neq \delta_j\in W\\[1mm]
\end{array}}^{}\Big(\prod\limits_{j=1,\delta_i\neq\delta_j\in W}^{i}\prod\limits_{t=0}^{r_j-1}(dim\mathcal{N}_{\delta_j}-t)\Big)\times\\[1mm]
&\times\Big(\prod\limits_{t=0}^{m-1-\sum r_i}(dim\mathcal{N}_{\delta_{i+1}}-t)\Big)dim\mathcal{N}_{\sum r_i\delta_i+(m-\sum r_i)\delta_{i+1}}.\\[1mm]
\end{array}\end{equation}

From \eqref{hm1} and \eqref{criteriya9} we finally get
\begin{equation}\label{eqmain}
dimZ^m(\mathcal{N},\mathcal{R})^{\mathcal{Q}}=dimC^m(\mathcal{N},\mathcal{R})^{\mathcal{Q}}-
rankS_m(\mathcal{R}).
\end{equation}

\

\subsection{Particular case: $dim\mathcal{N}_\alpha=1$ for any $\alpha\in W$.}

\

\

Here we focus our study on the condition $dim\mathcal{N}_\alpha=1$ for all $\alpha \in W,$ where
$\mathcal{N}_\alpha$ is root subspace with respect to action of $\mathcal{T}_{max}.$ Note that
the condition $dim\mathcal{N}_\alpha=1$ for all $\alpha \in W$ implies that nilradical of $\mathcal{R}$ is pure non-characteristically nilpotent Lie algebra and the condition {\bf A)} holds true. Then maximal solvable extension of such kind nilradical is isomorphic to $\mathcal{R}_{\mathcal{T}_{max}}$.

Equalities \eqref{criteriya9} and \eqref{criteriya9} have the following form:
$$dimC^m(\mathcal{N},\mathcal{R})^{\mathcal{Q}}=\sum\limits_{\delta_i\neq \delta_j\in W}dim\mathcal{N}_{\delta_1+\delta_2+\dots+\delta_{m}},$$
$$dimZ^m(\mathcal{N},\mathcal{R})^{\mathcal{Q}}=\sum\limits_{\delta_i\neq \delta_j\in W}dim\mathcal{N}_{\delta_1+\delta_2+\dots+\delta_{m}}-rankS_m(\mathcal{R}).$$

Consequently, for the space $B^m(\mathcal{N},\mathcal{R})^{\mathcal{Q}}$ we obtain
$$dimB^m(\mathcal{N},\mathcal{R})^{\mathcal{Q}}=dimC^{m-1}(\mathcal{N},\mathcal{R})^{\mathcal{Q}}-
dimZ^{m-1}(\mathcal{N},\mathcal{R})^{\mathcal{Q}}=rankS_{m-1}(\mathcal{R}),$$
which implies
\begin{equation}\label{hm4}
dimH^m(\mathcal{N},\mathcal{R})^\mathcal{Q}=\sum\limits_{\delta_i\neq \delta_j\in W} dim\mathcal{N}_{\delta_1+\delta_2+\dots+\delta_{m}}-rankS_m(\mathcal{R})-rankS_{m-1}(\mathcal{R}).
\end{equation}

\begin{lem}\label{hm2} Let $\mathcal{R}$ be a maximal solvable extension such that for any $\alpha+\beta+\gamma \in W$ with $\alpha, \beta, \gamma \in W$  imply either $\gamma=\beta$ or $\gamma=\alpha$. Then $dimH^m(\mathcal{N},\mathcal{R})^{\mathcal{Q}}=0$ for any $m\geq 0.$
\end{lem}
\begin{proof} Note that Theorem \ref{thmH1eq0} provides $dimH^0(\mathcal{N},\mathcal{R})^{\mathcal{Q}}=dimH^1(\mathcal{N},\mathcal{R})^{\mathcal{N}}=0.$
Also we note that, without loss of generality, from the condition of lemma on roots of $\mathcal{N}$ we can conclude that $\delta_1+\delta_2+\ldots+\delta_{m+1}\in W$ with $\delta_i\in W$ and $m\geq 2$ implies $\delta_{m}=\delta_{m+1},$ which means $$\sum\limits_{\delta_i\neq\delta_j}dim\mathcal{N}_{\delta_1}dim\mathcal{N}_{\delta_2}\cdots dim\mathcal{N}_{\delta_{m+1}}dim \mathcal{N}_{\delta_1+\dots+ \delta_{m+1}}=0.$$
Therefore, $\mathcal{S}_m=0$ for $m\geq 2.$

Since $rankS_{m}(\mathcal{R})\leq \mathcal{S}_m=0,$ we conclude $rankS_{m}(\mathcal{R})=0$ for $m\geq 2$.

So, due to Equality \eqref{hm4} we have
$$dimH^2(\mathcal{N},\mathcal{R})^{\mathcal{Q}}=\sum\limits_{\delta_i\neq \delta_j\in W} dim\mathcal{N}_{\delta_1+\delta_2}-dimB^2(\mathcal{N},\mathcal{R})^{\mathcal{Q}},$$
$$dimH^m(\mathcal{N},\mathcal{R})^{\mathcal{Q}}=\sum\limits_{\delta_i\neq \delta_j\in W} dim\mathcal{N}_{\delta_1+\delta_2+\dots+\delta_{m}}, \quad m\geq 3.$$

The proof of the lemma is completed by the following observations
$$\sum\limits_{\delta_i\neq \delta_j\in W}dim\mathcal{N}_{\delta_1+\delta_2}=dimB^2(\mathcal{N},\mathcal{R})^{\mathcal{Q}}=n-codim(\mathcal{N})$$
and $$\sum\limits_{\delta_i\neq \delta_j\in W}dim\mathcal{N}_{\delta_1+\delta_2+\dots+\delta_{m}}=0 \ \mbox{(because of the consitions of lemma)}.$$
\end{proof}

\begin{thm}\label{hm3} Let $\mathcal{R}$ be a maximal solvable extension such that for any $\alpha+\beta+\gamma \in W$ with $\alpha, \beta, \gamma \in W$  imply either $\gamma=\beta$ or $\gamma=\alpha$. Then $H(\mathcal{R},\mathcal{R})=0.$
\end{thm}
\begin{proof} The proof of the theorem is completed by the applications of the results of Theorem \ref{thmH1eq0}, Lemma \ref{hm2} and Equality \eqref{eq222}, that is, we derive $dimH^m(\mathcal{R},\mathcal{R})=0$ for all $m\geq 0$.
\end{proof}

\begin{thm}\label{2021q1} Let $\mathcal{R}$ be a maximal solvable extension such that
$\mathcal{N}^3=0$. Then $H(\mathcal{R},\mathcal{R})=0.$
\end{thm}
\begin{proof} The condition $\mathcal{N}^3=0$ implies that for any $m$ ($m\geq 3$) we have
$$\sum\limits_{\delta_i\neq \delta_j\in W} dim\mathcal{N}_{\delta_1+\delta_2+\dots+\delta_{m}}=rankS_m(\mathcal{R})=rankS_{m-1}=0.$$
Therefore, from Equality \eqref{hm4} we conclude $dim H^m(\mathcal{N},\mathcal{R})^Q=0$ for any $m\geq 3$.

Since $\mathcal{N}^3=0$, we get $dim \mathcal{N}^2=n-k$ and $\mathcal{S}_2=0$. Therefore,
$$\sum\limits_{\delta_1\neq \delta_2\in W}dim\mathcal{N}_{\delta_1+\delta_2}=n-k,
\quad rankS_2(\mathcal{R})=0.$$

Taking into account that $dimB^2(\mathcal{N},\mathcal{R})^{\mathcal{Q}}=n-k$ from Equality \eqref{hm4} we get
$$dimH^2(\mathcal{N},\mathcal{R})^{\mathcal{Q}}=\sum\limits_{\delta_1\neq \delta_2\in W} dim\mathcal{N}_{\delta_1+\delta_2}-dimB^2(\mathcal{N},\mathcal{R})^{\mathcal{Q}}=0.$$

Thanks to Theorem \ref{thmH1eq0} we have $dimH^0(\mathcal{N},\mathcal{R})^{\mathcal{Q}}=dimH^1(\mathcal{N},\mathcal{R})^{\mathcal{Q}}=0.$
Thus, we obtain $H^m(\mathcal{N},\mathcal{R})^{\mathcal{Q}}=0$ for any $m\geq 0$ and from Equality \eqref{eq222} we get $dim H^m(\mathcal{R},\mathcal{R})=0,\ m\geq 0$.
\end{proof}

\section{Application of Theorem \ref{hm3} to some solvable Lie algebras.}

In this section we give the applications of Theorem \ref{hm3}. Namely, we show that the results of the previous section much simplify the calculations of the cohomology groups.

\begin{exam}
In the paper \cite{Ancochea1} for a solvable Lie algebra $\mathcal{R}=\mathcal{N}\oplus \mathcal{Q}$ with the following table of multiplications:
$$\begin{array}{lll}
[e_i,e_1]=e_{i+1}, &2 \leq i \leq n_1, &\\[1mm]
[e_{n_1+\dots+n_{j}+i},e_1]=e_{n_1+\ldots+n_{j}+1+i},& 1\leq j\leq k-1,\ 2\leq i\leq n_{j+1}. \\[1mm]
[e_1,x_1]=e_{1}, &\\[1mm]
[e_i,x_1]=(i-2)e_{i}, &3 \leq i \leq n_1+1, &\\[1mm]
[e_{n_1+\dots+n_{j}+i},x_1]=(i-2)e_{n_1+\dots+n_{j}+i}, & 1\leq j\leq k-1,\ 3\leq i\leq n_{j+1}+1, &\\[1mm]
[e_i,x_2]=e_{i}, &2 \leq i \leq n_1+1, &\\[1mm]
[e_{n_1+\dots+n_{j}+i},x_{j+2}]=e_{n_1+\dots+n_{j}+i}, & 1\leq j\leq k-1,\ 2\leq i\leq n_{j+1}+1,  &\\[1mm]
\end{array}$$
where $\{e_1,\dots,e_{n_1+1} ,\dots,e_{n_1+n_2+1},\dots,e_{n_1+\dots + n_{k-1} +1},\dots, e_{n_1+\dots+n_k +1} \}$ is a basis of the nilradical $\mathcal{N}$ and $\{x_1, \dots, x_{k+1}\}$ is a basis of $\mathcal{Q}$, it was proved that $H^2(\mathcal{R},\mathcal{R})=0.$

From the table of multiplications of $\mathcal{R}$ we conclude that $\mathcal{R}\cong \mathcal{R}_{\mathcal{T}_{max}}$
and
$$\mathcal{N}_{\alpha_{n_1+2}}=\{e_{n_1+2}\}, \ \mathcal{N}_{(i-2)\alpha_1+\alpha_{n_1+2}}=\{e_{n_1+i} \}, \ 3\leq i\leq n_2+1 $$
$$\mathcal{N}_{\alpha_{n_1+\dots+n_j+1}}=\{e_{n_1+\dots+n_j+2}\}, \ \mathcal{N}_{(i-2)\alpha_1+\alpha_{n_1+\dots+n_j+2}}=\{e_{n_1+\dots+n_j+i} \}, \ 2\leq j\leq k-1,\ 3\leq i\leq n_{j+1}+1.$$

Since we are in the conditions of Theorem \ref{hm3}, we get not only $H^2(\mathcal{R},\mathcal{R})=0$, but also $H(\mathcal{R},\mathcal{R})=0$.
\end{exam}

\begin{exam} Let us consider the $(2n+2)$-dimensional solvable Lie algebra $r_{2n+2}$. As we already showed in Example \ref{com51} this algebra is isomorphic to the algebra $\mathcal{R}_{\mathcal{T}_{max}}= Q_{2n}\oplus \mathcal{T}_{max}.$

For root subspaces of $Q_{2n}$ we have
$$(Q_{2n})_{\alpha_{1}}=\{e_1\},\ (Q_{2n})_{\alpha_{2}}=\{e_2\}, \ (Q_{2n})_{(i-2)\alpha_1+\alpha_2}=\{e_{i} \}, \ 3\leq i\leq 2n-1,\ (Q_{2n})_{(2n-3)\alpha_1+2\alpha_2}=\{e_{2n}\}. $$

Since roots satisfy the conditions of Theorem \ref{hm3}, we conclude $H(r_{2n+2},r_{2n+2})=0$.
\end{exam}

\begin{exam} For the solvable Lie algebra $R(\bar{N})$ mentioned in Example \ref{com52}
we have $R(\bar{N})= \bar{N}\oplus \mathcal{T}_{max}$ and
$$\bar{N}_{\alpha_{1}}=\{e_n\},\ \bar{N}_{\alpha_{2}}=\{e_{n-2}\},\ \bar{N}_{\alpha_{3}}=\{e_{n-1}\},\ \bar{N}_{(i-2)\alpha_1+\alpha_2}=\{e_{n-i} \}, \ 3\leq i\leq n-1. $$
Again, because the roots satisfy the conditions of Theorem \ref{hm3} we get $H(R(\bar{N}),R(\bar{N}))=0$.
\end{exam}

\begin{exam} Consider the solvable Lie algebra $R(n_{n,1})=n_{n,1}\oplus \mathcal{T}_{max}$ from Example \ref{com53}. Then for its nilradical $n_{n,1}$ we have roots decomposition with
$$(n_{n,1})_{\alpha_{1}}=\{e_n\},\ (n_{n,1})_{\alpha_{2}}=\{e_{n-1}\},\ (n_{n,1})_{(i-1)\alpha_1+\alpha_2}=\{e_{n-i} \}, \ 2\leq i\leq n-1. $$

Thanks to Theorem \ref{hm3} we conclude $H(R(n_{n,1}),R(n_{n,1}))=0$.
\end{exam}

Below we give examples of solvable Lie algebras $\mathcal{R}_{\mathcal{T}_{max}}$ such that the roots of their nilradicals do not satisfy the condition of Theorem \ref{hm3}.

\begin{exam} Let us consider a solvable Lie algebra $\mathcal{R}_{\mathcal{T}_{max}}$ with the following table of multiplications:
$$
\left\{\begin{array}{llllll}\label{leger4}
[e_1,e_2]=e_4, & [e_1,e_3]=e_5,& [e_2,e_3]=e_6,& \\[1mm]
[e_4,e_3]=e_7, &[e_6,e_1]=-e_7,& \\[1mm]
[e_1,x_1]=e_1, & [e_4,x_1]=e_4,& [e_5,x_1]=e_5,& [e_7,x_1]=e_7,\\[1mm]
[e_2,x_2]=e_2, & [e_4,x_2]=e_4,& [e_6,x_2]=e_6,& [e_7,x_2]=e_7,\\[1mm]
[e_3,x_3]=e_3, & [e_5,x_3]=e_5,& [e_6,x_3]=e_6,& [e_7,x_3]=e_7.\\[1mm]
\end{array}\right.
$$
For its nilradical we have
$$\mathcal{N}_{\alpha_1}=\{e_1\},\  \mathcal{N}_{\alpha_2}=\{e_2\},\  \mathcal{N}_{\alpha_3}=\{e_3\},\  \mathcal{N}_{\alpha_1+\alpha_2}=\{e_4\},$$
$$\mathcal{N}_{\alpha_1+\alpha_3}=\{e_5\}
\ \mathcal{N}_{\alpha_2+\alpha_3}=\{e_6\},\  \mathcal{N}_{\alpha_1+\alpha_2+\alpha_3}=\{e_7\}.$$
Clearly, the roots of the nilradical do not satisfy the condition of Theorem \ref{hm3}. The computations lead to $dimH^2(\mathcal{R},\mathcal{R})=1.$ Therefore, $H(\mathcal{R},\mathcal{R})\neq 0.$
\end{exam}

\begin{exam} Let us consider a Borel subalgebra $B$ of a semi-simple Lie algebra over a field of characteristic $0.$ It is not difficult to check that the roots of nilradical of $B$ do not satisfy the condition of Theorem \ref{hm3}, while in the paper \cite{Leger2} it was proved that $H(B, B) = 0.$
\end{exam}

{\bf Acknowledgments:} The authors are grateful to Professors V. Gorbatsevich and L. \v{S}nobl for their valuable comments which corrected the original version of the paper.

\end{document}